\documentclass{article}
\usepackage[latin1]{inputenc}
\usepackage{amsmath}
\usepackage{amsfonts,amssymb,amsthm,todonotes,enumerate,bbm,algpseudocode,algorithm,etoolbox,pgfplots,xassoccnt,xr-hyper,hyperref,fullpage,booktabs,comment,forloop}
\presetkeys{todonotes}{inline}{}
\pgfplotsset{compat=newest}

\makeatletter
\expandafter\patchcmd\csname\string\algorithmic\endcsname{\itemsep\z@}{\itemsep=3pt}{}{}
\makeatother

\usepackage{float}
\newfloat{algorithm}{t}{lop}

\pgfplotsset{
	discard if not/.style 2 args={
		x filter/.code={
			\edef\tempa{\thisrow{#1}}
			\edef\tempb{#2}
			\ifx\tempa\tempb
			\else
			
			\fi
		}
	}
}

\newtheorem{lemma}{Lemma}
\newtheorem{theorem}{Theorem}
\newtheorem{corollary}{Corollary}
\theoremstyle{definition}
\newtheorem{definition}{Definition}

\DeclareCoupledCountersGroup{figtheoexadef}
\DeclareCoupledCounters[name=figtheoexadef]{theorem,lemma,definition,corollary}

\author{Riley Badenbroek \and Etienne de Klerk}
\title{Simulated annealing with hit-and-run for convex optimization: rigorous complexity analysis and practical perspectives for copositive programming}

\DeclareMathOperator{\diff}{d\!}

\DeclareMathOperator{\smat}{smat}
\DeclareMathOperator{\svec}{svec}
\DeclareMathOperator{\Diag}{Diag}

\algnewcommand\Input{\item[\bf Input:] }
\algnewcommand\Output{\item[\bf Output:] }
\newcommand\algmultiline[1]{\parbox[t]{\dimexpr\textwidth-\leftmargin-\labelsep-\labelwidth}{#1\strut}}
\newcommand\algmultilinetwo[1]{\parbox[t]{\dimexpr\textwidth-\leftmargin-\labelsep-\labelwidth-\leftmargin}{#1\strut}}

\pgfplotsset{yticklabel style={text width=1.5em,align=right}, ylabel near ticks}

\begin{document}

\maketitle

\begin{abstract}
	We give a rigorous complexity analysis of the simulated annealing algorithm by Kalai and Vempala
[Math of OR 31.2 (2006): 253-266] using the type of temperature update
 suggested by Abernethy and Hazan [arXiv 1507.02528v2, 2015].
 The algorithm only assumes a membership oracle of the feasible set, and we prove that it returns a solution in polynomial time
  which is near-optimal with high probability.
	Moreover, we propose a number of modifications to improve the practical performance of this method,
 and present some numerical results for test problems from copositive programming.
 \end{abstract}

\noindent
{ \bf Keywords:} {Simulated annealing, convex optimization, hit-and-run sampling, semidefinite and copositive programming}

\section{Introduction}
Let $K \subseteq \mathbb{R}^n$ be a convex body, and suppose that only a membership oracle of $K$ is available. Fix a vector $c \in \mathbb{R}^n$, and let $\langle \cdot, \cdot \rangle$ be an inner product on $\mathbb{R}^n$. We are interested in the problem
\begin{equation}
	\min_{x \in K} \langle c, x \rangle.
	\label{eq:MainProblem}
\end{equation}
One approach to solve problems of this type is using simulated annealing, a paradigm of randomized algorithms for general optimization introduced by Kirkpatrick et al. \cite{kirkpatrick1983optimization}. It features a so-called temperature parameter that decreases during the run of the simulated annealing algorithm. At high temperatures, the method explores the feasible set relatively freely, also moving to solutions that have worse objective values than the current point. As the temperature decreases, so does the probability that a solution with a worse objective value is accepted.
Kalai and Vempala \cite{kalai2006simulated} showed that, for convex optimization, a polynomial-time simulated annealing algorithm exists. Specifically, their algorithm returns a feasible solution that is near-optimal with high probability in polynomial time.

Abernethy and Hazan \cite{abernethy2015faster} recently clarified that Kalai and Vempala's algorithm is in some sense equivalent to an interior point method. In general, interior point methods for convex bodies require a so-called self-concordant barrier of the feasible set. It was shown by Nesterov and Nemirovskii \cite{nesterov1994interior} that every open convex set that does not contain an affine subspace is the domain of a self-concordant barrier, known as the universal barrier. However, it is not known how to compute the gradients and Hessians of this barrier in general.

The interior point method to which Kalai and Vempala's algorithm corresponds uses the so-called entropic barrier over $K$, to be defined later. This barrier was introduced by Bubeck and Eldan \cite{bubeck2018entropic}, who also established the self-concordance of the barrier and analyzed its complexity parameter $\vartheta$.

Drawing on the connection to interior point methods, Abernethy and Hazan \cite{abernethy2015faster} proposed a new temperature schedule for Kalai and Vempala's algorithm. This schedule does not depend on the dimension $n$ of the problem, but on the complexity parameter $\vartheta$ of the entropic barrier.
Our goal is to prove in detail that simulated annealing with this new temperature schedule returns a solution in polynomial time which is near-optimal with high probability. Moreover, we aim to investigate the practical applicability of this method.

\subsection{Algorithm Statement}
Central to Kalai and Vempala's algorithm is a family of exponential-type probability distributions known as Boltzmann distributions.
\begin{definition}
	Let $K \subseteq \mathbb{R}^n$ be a convex body, and let $\theta \in \mathbb{R}^n$. Let $\langle \cdot, \cdot \rangle$ be the Euclidean inner product. Then, the \emph{Boltzmann distribution with parameter $\theta$} is the probability distribution supported on $K$ having density with respect to the Lebesgue measure proportional to $x \mapsto e^{\langle \theta, x \rangle}$.
\end{definition}
Throughout this work, we will use $\Sigma(\theta)$ to refer to the covariance matrix of the Boltzmann distribution with parameter $\theta \in \mathbb{R}^n$. If $\langle \cdot, \cdot \rangle$ is some reference inner product, then $\langle x, y \rangle_\theta := \langle x, \Sigma(\theta) y \rangle$ for any $\theta \in \mathbb{R}^n$. Moreover, let $\| \cdot \|_\theta$ denote the norm induced by the inner product $\langle \cdot, \cdot \rangle_\theta$.

As mentioned above, Abernethy and Hazan's temperature schedule depends on the complexity parameter of the entropic barrier. This function is defined as follows.
\begin{definition}
	Let $K \subseteq \mathbb{R}^n$ be a convex body. Define the function $f: \mathbb{R}^n \to \mathbb{R}$ as $f(\theta) = \ln \int_K e^{\langle \theta, x \rangle} \diff x$. Then, the convex conjugate $f^*$ of $f$,
	\begin{equation*}
		f^*(x) = \sup_{\theta \in \mathbb{R}^n} \left\{ \langle \theta, x \rangle - f(\theta) \right\},
	\end{equation*}
	is called the \emph{entropic barrier} for $K$.
\end{definition}
Bubeck and Eldan \cite{bubeck2018entropic} showed that $f^*$ is a self-concordant barrier for $K$ with complexity parameter $\vartheta \leq n +o(n)$. The complexity parameter of $f^*$ is
\begin{equation}
	\vartheta := \sup_{x \in K} \langle Df^*(x), [D^2 f^*(x)]^{-1} Df^*(x) \rangle = \sup_{\theta \in \mathbb{R}^n} \langle \theta, \Sigma(\theta) \theta \rangle = \sup_{\theta \in \mathbb{R}^n} \| \theta \|_\theta^2,
	\label{eq:ComplexityParameter}
\end{equation}
where we refer the reader to \cite[inequality \eqref{ext1-eq:UpperBoundLocalTheta}]{badenbroek2018complexity} for details.

The procedure Kalai and Vempala \cite{kalai2006simulated} use to generate samples on $K$ is called \emph{hit-and-run} sampling. This Markov Chain Monte Carlo method was introduced by Smith \cite{smith1984efficient} to sample from the uniform distribution over a bounded convex set. Later, it was generalized to absolutely continuous distributions (see for example B\'elisle et al. \cite{belisle1993hit}).
The details of hit-and-run sampling are given in Algorithm \ref{alg:HitAndRun}.

\begin{algorithm}[!ht]
	\begin{algorithmic}[1]
		\Input Oracle for a convex body $K \subseteq \mathbb{R}^n$; probability distribution $\mu$ to sample from
(i.e. the target distribution); covariance matrix $\Sigma \in \mathbb{R}^{n \times n}$; starting point $x \in K$;
number of hit-and-run steps $\ell \in \mathbb{N}$.
		\Statex
		
		\State $X_0 \gets x$
		\For{$i \in \{1, ..., \ell \}$}
				\State Sample direction $D_i$ from a $\mathcal{N}(0, \Sigma)$-distribution
				\State Sample $X_i$ from the univariate distribution $\mu$ restricted to $K \cap \{X_{i-1} + t D_i : t \in \mathbb{R} \}$
		\label{line:SampleLineSegment}
		\EndFor
		\State \Return $X_\ell$
	\end{algorithmic}
	\caption{The hit-and-run sampling procedure}
	\label{alg:HitAndRun}
\end{algorithm}

The algorithm by Kalai and Vempala \cite{kalai2006simulated} that uses a temperature schedule of the type introduced by Abernethy and Hazan \cite{abernethy2015faster} is now given in Algorithm \ref{alg:KalaiVempala}.
  Note in particular that the temperature parameter in iteration $k$ of Algorithm \ref{alg:KalaiVempala} is given by $T_k$.

\begin{algorithm}[ht!]
	\caption{Algorithm by Kalai and Vempala \cite{kalai2006simulated} using temperature schedule of type introduced by Abernethy and Hazan \cite{abernethy2015faster}}
	\label{alg:KalaiVempala}
	\begin{algorithmic}[1]
		\Input
		unit vector $c \in \mathbb{R}^n$; membership oracle $\mathcal{O}_K: \mathbb{R}^n \to \{0,1\}$ of a convex body $K$; radius $R$ of Euclidean ball containing $K$;
		complexity parameter $\vartheta \leq n + o(n)$ of the entropic barrier over $K$;
		$x_0 \in K$ drawn randomly from the uniform distribution over $K$; update parameter $\alpha > 1+ 1/\sqrt{\vartheta}$;
		number of phases $m \in \mathbb{N}$; number of hit-and-run steps $\ell \in \mathbb{N}$; number of covariance update samples $ N \in \mathbb{N}$; approximation $\widehat{\Sigma}(0)$ of
$\Sigma(0)$ satisfying $\tfrac{1}{2} v^\top \widehat{\Sigma}(0) v \leq v^\top \Sigma(0) v \leq 2 v^\top \widehat{\Sigma}(0) v$ for all $v \in \mathbb{R}^n$; starting temperature $T_0 = R$.
		\Output $x_m$ such that $\langle c, x_m \rangle - \min_{x \in K} \langle c, x \rangle \leq \bar{\epsilon}$ at terminal iteration $m$.
		\Statex
		
		\State $X_0 \gets x_0$
		\State $\theta_0 = 0$
		\For{$k \in \{1, ..., m\}$}
		\State  $T_{k} \gets \left(1 - \frac{1}{\alpha \sqrt{\vartheta}}\right)T_{k-1}$
		\State $\theta_k \gets -c / T_k$
		\State \algmultiline{Generate $X_k$ by applying hit-and-run sampling to the Boltzmann distribution with parameter $\theta_k$, starting the walk from $X_{k-1}$, taking $\ell$ steps, drawing directions from a $\mathcal{N}(0, \widehat{\Sigma}(\theta_{k-1}))$-distribution}
		\For{$j \in \{1, ..., N\}$} 
			\State \algmultilinetwo{Generate $Y_{jk}$ by applying hit-and-run sampling to the Boltzmann distribution with parameter $\theta_k$, starting the walk from $X_{k-1}$, taking $\ell$ steps, drawing directions from a $\mathcal{N}(0, \widehat{\Sigma}(\theta_{k-1}))$-distribution}
		\EndFor
		\State $\widehat{\Sigma}(\theta_k) \gets \frac{1}{N} \sum_{j=1}^{N} Y_{jk} Y_{jk}^\top - \left( \frac{1}{N} \sum_{j=1}^{N} Y_{jk} \right) \left( \frac{1}{N} \sum_{j=1}^{N} Y_{jk} \right)^\top$ \label{line:EmpiricalCovariance}
		\EndFor
		\State \Return $X_m$
	\end{algorithmic}
\end{algorithm}

In their original paper, Kalai and Vempala \cite{kalai2006simulated} show that the algorithm returns a near-optimal solution
 with high probability, for the temperature update (cf.\ line 4 in Algorithm \ref{alg:KalaiVempala})
 \begin{equation}\label{temp update Kalai Vempala}
   T_{k} = \left(1-\frac{1}{\sqrt{n}}\right)T_{k-1},
 \end{equation}
 in $m = O^*(\sqrt{n})$ iterations,
  where $O^*$ suppresses polylogarithmic terms in the problem parameters. Abernethy and Hazan \cite{abernethy2015faster} propose the alternative
  temperature update
  \begin{equation}\label{AH temp schedule}
     T_{k} =  \left(1 - \frac{1}{4 \sqrt{\vartheta}}\right)T_{k-1},
  \end{equation}
  which will lead to
   $m = O^*(\sqrt{\vartheta})$ iterations. As noted above, we have $\vartheta \leq n +o(n)$ in general, but
   it is not currently known if $\vartheta < n$ for any convex bodies.
   In particular, the temperature update \eqref{AH temp schedule} only improves on \eqref{temp update Kalai Vempala} if $\vartheta < n/16$, which is not known to hold for any convex body.
      We show in Appendix \ref{appendix:complexity parameter ball} to this paper that, for the Euclidean unit ball in $\mathbb{R}^n$, numerical evidence suggests that $\vartheta = (n+1)/2$.
      We therefore consider a variation on the temperature schedule \eqref{AH temp schedule} suggested by Abernethy and Hazan, namely
    \begin{equation}
    \label{AH type temp schedule}
    T_{k} = \left(1 - \frac{1}{\alpha \sqrt{\vartheta}}\right)T_{k-1} \mbox{ for some } \alpha > 1+ 1/\sqrt{\vartheta},
    \end{equation}
     which corresponds to \eqref{AH temp schedule} when $\alpha = 4$, but gives larger temperature reductions when $\alpha < 4$.
     We will refer to \eqref{AH type temp schedule} as
     Abernethy-Hazan-type temperature updates. If $\vartheta < n$, this results in a larger temperature decrease than
     the Kalai and Vempala \cite{kalai2006simulated} scheme \eqref{temp update Kalai Vempala}, for a suitable choice of the parameter $\alpha$.

\subsection{Outline of this paper}
    Abernethy and Hazan \cite{abernethy2015faster} do not give a rigorous analysis of the temperature schedule \eqref{AH temp schedule} in their paper,
     only a sketch of the proof.
    In this paper we provide the full details for the more general schedule \eqref{AH type temp schedule}.
    We start with a review of useful facts on probability distributions in Section \ref{sec:Preliminaries on probability distributions},
    followed by a section on mixing conditions for hit-and-run sampling in Section \ref{sec:Mixing Conditions}. In Section \ref{sec:Proof of Convergence}
     we give the main analysis
    on the rate of convergence of Algorithm \ref{alg:KalaiVempala}. In doing so, we also provide some details that were omitted in the
    original work by Kalai and Vempala \cite{kalai2006simulated}, that concerns the application of a theorem by Rudelson \cite{rudelson1999random}.
    In the remainder of the paper, we discuss the perspectives for practical computation with Algorithm \ref{alg:KalaiVempala}.
    In Section \ref{sec:Numerical Examples on the Doubly Nonnegative Cone}, we look at the behavior of hit-and-run sampling for optimization over
    the doubly nonnegative cone, and suggest some heuristic improvements to obtain speed-up.
    We then evaluate the resulting algorithm on problems from copositive programming (due to Berman et al.\ \cite{berman2015open})
    in Section \ref{sec:Numerical Examples on the Copositive Cone}.

\section{Preliminaries on probability distributions}
\label{sec:Preliminaries on probability distributions}
Below, we will define some tools necessary for the analysis of Kalai and Vempala's algorithm with the type of temperature schedule by Abernethy and Hazan. We start with the notion of isotropy.
\begin{definition}
	\label{def:NearIsotropy}
	Let $(K, \mathcal{E})$ be a measurable space with $K \subseteq \mathbb{R}^n$, and $\epsilon \geq 0$. Let $\langle \cdot, \cdot \rangle$ be the Euclidean inner product. A probability distribution $\mu$ over $(K, \mathcal{E})$ is \emph{$(1+\epsilon)$-isotropic} if for every $v \in \mathbb{R}^{n}$,
	\begin{equation*}
		\frac{1}{1+\epsilon} \| v \|^2 \leq \int_{K} \langle v, x - \mathbb{E}_\mu[X]\rangle^2 \diff \mu(x) \leq (1+\epsilon) \|v\|^2.
	\end{equation*}
\end{definition}

Moreover, we will need two measures of divergence between probability distributions. Before we can define them, we recall the definition of absolute continuity.
\begin{definition}
	Let $(K, \mathcal{E})$ be a measurable space, and let $\nu$ and $\mu$ be measures on this space. Then, $\nu$ is \emph{absolutely continuous} with respect to $\mu$ if $\mu(A) = 0$ implies $\nu(A) = 0$ for all $A \in \mathcal{E}$.
	We write this property as $\nu \ll \mu$.
\end{definition}

The first measure of divergence between probability distributions is the $L_2$-norm.
\begin{definition}
	Let $(K, \mathcal{E})$ be a measurable space. Let $\nu$ and $\mu$ be two probability distributions over this space, such that $\nu \ll \mu$. Then,
	the \emph{$L_2$-norm of $\nu$ with respect to $\mu$} is
	\begin{equation*}
		\|\nu / \mu \| := \int_{K} \frac{\diff \nu}{\diff \mu} \diff \nu = \int_{K} \left(\frac{\diff \nu}{\diff \mu} \right)^2 \diff \mu,
	\end{equation*}
	where $\frac{\diff \nu}{\diff \mu}$ is the Radon-Nikodym derivative of $\mu$ with respect to $\nu$.
\end{definition}
It is easily shown that the $L_2$-norm is invariant under invertible affine transformations.
\begin{lemma}
	\label{lemma:TransformedL2Norm}
	Let $(K, \mathcal{E})$ be a measurable space, and suppose $K \subseteq \mathbb{R}^n$. Let $\nu$ and $\mu$ be two probability measures over this space having densities $h_\nu$ and $h_\mu$ with respect to the Lebesgue measure, respectively. Let $A: \mathbb{R}^n \to \mathbb{R}^n$ be an invertible linear operator and write $\overline{K} := \{ Ax : x \in K \}$. Define the transformed probability density $\overline{h}_\nu: \overline{K} \to \mathbb{R}_+$ by $\overline{h}_\nu(y) = \det(A^{-1}) h_\nu(A^{-1} y)$, and similarly for $\overline{h}_\mu$. Denote their induced distributions over $\overline{K}$ by $\overline{\nu}$ and $\overline{\mu}$. Then,
	\begin{equation*}
		\| \nu / \mu \| = \| \overline{\nu} / \overline{\mu} \|.
	\end{equation*}
\end{lemma}

The second measure of divergence between probability distributions is the total variation distance.
\begin{definition}
	\label{def:TotalVariationDistance}
	Let $(K, \mathcal{E})$ be a measurable space. For two probability distributions $\mu$ and $\nu$ over this space, their \emph{total variation distance} is
	\begin{equation*}
		\|\mu - \nu \| := \sup_{A \in \mathcal{E}} | \mu(A) - \nu(A)|.
	\end{equation*}
\end{definition}
A useful property of the total variation distance is that it allows coupling of random variables, in the sense of the following lemma.
\begin{lemma}[e.g. Proposition 4.7 in \cite{levin2017markov}]
	Let $X$ be a
	random variable on $K \subseteq \mathbb{R}^n$
	with distribution $\mu$, and let $\nu$ be a different probability distribution on $K$.
	If $\|\mu - \nu\| = \alpha$, there exists a random variable $Y$ on $K$ distributed according to $\nu$ such that the joint distribution of $X$ and $Y$ satisfies $\mathbb{P}\{X = Y\} = 1-\alpha$.
	\label{lemma:DivineIntervention}
\end{lemma}

\section{Mixing Conditions}
\label{sec:Mixing Conditions}
The main purpose of this section is to study how we can use hit-and-run sampling to approximate covariance matrices of Boltzmann distributions as in Line \ref{line:EmpiricalCovariance} of Algorithm \ref{alg:KalaiVempala}. Our first step is showing that hit-and-run indeed mixes in our setting. While Theorem \ref{ext1-thm:MixingTimeSC} in \cite{badenbroek2018complexity} gives conditions under which it can be done, it depends on an approximation of an \emph{inverse} covariance matrix. In Kalai and Vempala's algorithm, it is more convenient to maintain that we have an approximation of a covariance matrix, not necessarily its inverse. We therefore state the following mixing theorem.
\begin{theorem}
	\label{thm:MixingTimeSCCov}
	Let $K \subseteq \mathbb{R}^n$ be a convex body, and let $\langle \cdot, \cdot \rangle$ be the Euclidean inner product. Let $q > 0$, and $\theta_0, \theta_1 \in \mathbb{R}^n$ such that $\Delta \theta := \max\{ \| \theta_1 - \theta_0 \|_{\theta_0}, \| \theta_1 - \theta_0 \|_{\theta_1} \} < 1$. Define $\Delta\theta_0 := \| \theta_1 - \theta_0 \|_{\theta_0}$. Pick $\epsilon \geq 0$, and suppose we have an invertible matrix $\widehat{\Sigma}(\theta_0)$ such that
	\begin{equation}
		\frac{1}{1+\epsilon} v^\top \widehat{\Sigma}(\theta_0) v \leq v^\top \Sigma(\theta_0) v \leq (1+\epsilon) v^\top \widehat{\Sigma}(\theta_0) v \qquad \forall v \in \mathbb{R}^n. \label{eq:HitAndRunCovConditionGeneral}
	\end{equation}
	Consider a hit-and-run random walk applied to the Boltzmann distribution $\mu$ with parameter $\theta_1$ from a random starting point drawn from a Boltzmann distribution with parameter $\theta_0$.
	Let $\nu^{(\ell)}$ be the distribution of the hit-and-run point after $\ell$ steps of hit-and-run sampling applied to $\mu$, where the directions are drawn from a $\mathcal{N}(0,\widehat{\Sigma}(\theta_0))$-distribution. Then, after
	\begin{equation}
		\label{eq:WalkLength}
		\ell = \left\lceil \frac{16384 e^2 10^{30} n^3 (1+\epsilon)^2 }{(1-\Delta\theta)^4 \exp(4 \Delta\theta)}
		\ln^2 \left( \frac{256 \exp(-2\Delta\theta_0) n\sqrt{n}(1+\epsilon) }{(1-\Delta\theta_0)^2 (1- \Delta\theta)^2 \exp(2 \Delta\theta ) q^2} \right)
		\ln^3 \left( \frac{2 \exp(-2 \Delta\theta_0)}{(1-\Delta\theta_0)^2 q^2} \right) \right\rceil,
	\end{equation}
	hit-and-run steps, we have $\| \nu^{(\ell)} - \mu \| \leq q$.
\end{theorem}
\begin{proof}
	To show the claim, we will prove that conditions \eqref{ext1-con:ContainedBall}-\eqref{ext1-con:LocalVariance} from Corollary \ref{ext1-cor:MixingTime} in \cite{badenbroek2018complexity} are satisfied. As in Kalai and Vempala \cite{kalai2006simulated}, the three conditions will follow if the distribution to sample from, i.e. $\mu$, is near-isotropic after $K$ is transformed by $\widehat{\Sigma}(\theta_0)^{-1/2}$.
	
	Denote the Boltzmann distributions over $K$ with parameter $\theta_0$ by $\nu$.
	With $v = \widehat{\Sigma}(\theta_0)^{-1/2} u$, \eqref{eq:HitAndRunCovConditionGeneral} is equivalent to $\frac{1}{1+\epsilon} \| u \|^2 \leq u^\top \widehat{\Sigma}(\theta_0)^{-1/2} \Sigma(\theta_0) \widehat{\Sigma}(\theta_0)^{-1/2} u \leq (1+\epsilon) \| u \|^2$.
	In other words,
	\begin{equation*}
		\frac{1}{1+\epsilon} \| u \|^2 \leq \frac{\int_K \left[u^\top \widehat{\Sigma}(\theta_0)^{-1/2} (x - \mathbb{E}_{\theta_0}[X])\right]^2 e^{\langle \theta_0, x \rangle} \diff x}{\int_K e^{\langle \theta_0, x \rangle} \diff x} \leq (1+\epsilon) \| u \|^2.
	\end{equation*}
	With a change of variables $y = \widehat{\Sigma}(\theta_0)^{-1/2} x$, this result is equivalent to
	\begin{equation*}
		\frac{1}{1+\epsilon} \| u \|^2 \leq \frac{\int_{\widehat{\Sigma}(\theta_0)^{-1/2} K} \left[u^\top (y - \widehat{\Sigma}(\theta_0)^{1/2} \mathbb{E}_{ \theta_0}[X])\right]^2 e^{\langle \widehat{\Sigma}(\theta_0)^{1/2} \theta_0,  y \rangle} \diff y}{\int_{\widehat{\Sigma}(\theta_0)^{-1/2} K} e^{\langle \widehat{\Sigma}(\theta_0)^{1/2} \theta_0, y \rangle} \diff y} \leq (1+\epsilon) \| u \|^2,
	\end{equation*}
	where $\widehat{\Sigma}(\theta_0)^{1/2} \mathbb{E}_{\theta_0}[X]$ equals
	\begin{equation*}
		\widehat{\Sigma}(\theta_0)^{-1/2} \mathbb{E}_{\theta_0}[X] = \frac{\int_K \widehat{\Sigma}(\theta_0)^{-1/2} x e^{\langle \theta_0, x \rangle} \diff x}{ \int_K e^{\langle \theta_0, x \rangle} \diff x }
		= \frac{\int_{\widehat{\Sigma}(\theta_0)^{-1/2} K} y e^{\langle \widehat{\Sigma}(\theta_0)^{1/2} \theta_0, y \rangle} \diff y}{ \int_{\widehat{\Sigma}(\theta_0)^{-1/2} K} e^{\langle \widehat{\Sigma}(\theta_0)^{1/2} \theta_0, y \rangle} \diff y }.
	\end{equation*}
	Hence, \eqref{eq:HitAndRunCovConditionGeneral} is equivalent to the statement that the Boltzmann distribution with parameter $\widehat{\Sigma}(\theta_0)^{1/2} \theta_0$ over $\widehat{\Sigma}(\theta_0)^{-1/2} K$ is in $(1+\epsilon)$-isotropic position. Let us refer to this distribution as $\overline{\nu}$, and to the Boltzmann distribution over the same body with parameter $\widehat{\Sigma}(\theta_0)^{1/2} \theta_1$ as $\overline{\mu}$.
	By Lemma 4.3 from Kalai and Vempala \cite{kalai2006simulated},
	$\overline{\mu}$ is $t$-isotropic with
	\begin{equation*}
		t = 16 (1+\epsilon) \max\{ \| \overline{\nu} / \overline{\mu} \|, \| \overline{\mu} / \overline{\nu} \| \}
		= 16 (1+\epsilon) \max\{ \| \nu / \mu \|, \| \mu / \nu \| \},
	\end{equation*}
	where the final equality is due to Lemma \ref{lemma:TransformedL2Norm}. To upper bound the $L_2$-norms between $\mu$ and $\nu$, we use the fact that they are Boltzmann distributions.
	By Lemma \ref{ext1-lemma:L2UpperBound} in \cite{badenbroek2018complexity}, the value of $t$ can therefore be bounded by
	\begin{equation*}
		t \leq 16 (1+\epsilon) \frac{\exp(-2 \max\{ \| \theta_1 - \theta_0 \|_{\theta_0}, \| \theta_1 - \theta_0 \|_{\theta_1} \} )}{(1 - \max\{ \| \theta_1 - \theta_0 \|_{\theta_0}, \| \theta_1 - \theta_0 \|_{\theta_1} \} )^2}
		= 16 (1+\epsilon) \frac{\exp(-2 \Delta\theta )}{(1 - \Delta\theta)^2}.
	\end{equation*}
	
	Kalai and Vempala \cite[Lemma 4.2]{kalai2006simulated} show that the level set of $\overline{\mu}$ with probability $\frac{1}{8}$ contains a Euclidean ball of radius $\frac{1}{8e \sqrt{t}}$ and $\mathbb{E}_{\overline{\mu}}[\| Y - \mathbb{E}_{\overline{\mu}}[Y] \|^2] \leq t n$. Transforming back to $K$, we have that the level set of $\mu$ with probability $\frac{1}{8}$ contains a $\| \cdot \|_{\widehat{\Sigma}(\theta_0)^{-1}}$-ball with radius $\frac{1}{8e \sqrt{t}}$ and $\mathbb{E}_{\mu}[\| X - \mathbb{E}_{\mu}[X] \|^2_{\widehat{\Sigma}(\theta_0)^{-1}}] \leq tn$.
	Since Lemma \ref{ext1-lemma:L2UpperBound} in \cite{badenbroek2018complexity} can be used to upper bound $\| \nu / \mu\|$, all conditions for Corollary \ref{ext1-cor:MixingTime} in \cite{badenbroek2018complexity} are satisfied, which proves the statement.
\end{proof}

In iteration $k$ of Algorithm \ref{alg:KalaiVempala}, one uses the temperature update \eqref{AH type temp schedule},
as well as sampling from the Boltzmann distribution with parameter $\theta_k := -c / T_k$. For these parameters, we can derive the following corollary to Theorem \ref{thm:MixingTimeSCCov}.

\begin{corollary}
	\label{cor:MixingTimeSA}
	Let $K \subseteq \mathbb{R}^n$ be a convex body, and let $\langle \cdot, \cdot \rangle$ be the Euclidean inner product. Let $q > 0$, and $\theta_0, \theta_1 \in \mathbb{R}^n$ such that $\theta_1 = \theta_0/(1 - \frac{1}{\alpha\sqrt{\vartheta}})$, where $\vartheta$ is the complexity parameter of the entropic barrier over $K$, and $\alpha > 1+1/\sqrt{\vartheta}$.
	Pick $\epsilon \geq 0$, and suppose we have an invertible matrix $\widehat{\Sigma}(\theta_0)$ such that
	\begin{equation}
		\frac{1}{1+\epsilon} v^\top \widehat{\Sigma}(\theta_0) v \leq v^\top \Sigma(\theta_0) v \leq (1+\epsilon) v^\top \widehat{\Sigma}(\theta_0) v \qquad \forall v \in \mathbb{R}^n.
		\label{eq:HitAndRunCovCondition}
	\end{equation}
	Consider a hit-and-run random walk
	applied to the Boltzmann distribution $\mu$ with parameter $\theta_1$ from a random starting point drawn from a Boltzmann distribution with parameter $\theta_0$.
	Let $\nu^{(\ell)}$ be the distribution of the hit-and-run point after $\ell$ steps of hit-and-run sampling applied to $\mu$, where the directions are drawn from a $\mathcal{N}(0,\widehat{\Sigma}(\theta_0))$-distribution. Then, after
	\begin{align}
		\begin{aligned}
			\ell &= \Bigg\lceil \frac{16384 e^2 10^{30} n^3 (1+\epsilon)^2 (\alpha \sqrt{\vartheta} -1)^4 }{((\alpha-1)\sqrt{\vartheta} - 1)^4 \exp(4 \sqrt{\vartheta} / (\alpha \sqrt{\vartheta} - 1))}
			\ln^2 \left( \frac{256 \exp(-2\sqrt{\vartheta} / (\alpha \sqrt{\vartheta} - 1)) n\sqrt{n}(1+\epsilon) (\alpha \sqrt{\vartheta} - 1)^4 }{((\alpha-1)\sqrt{\vartheta} - 1)^4 \exp(2 \sqrt{\vartheta} / (\alpha \sqrt{\vartheta} - 1) ) q^2} \right)\\
			&\times \ln^3 \left( \frac{2 \exp(-2 \sqrt{\vartheta} / (\alpha \sqrt{\vartheta} - 1)) (\alpha \sqrt{\vartheta} -1)^2 }{((\alpha-1)\sqrt{\vartheta} - 1)^2 q^2} \right) \Bigg\rceil,
		\end{aligned}
		\label{eq:WalkLengthSA}
	\end{align}
	hit-and-run steps, we have $\| \nu^{(\ell)} - \mu \| \leq q$.
\end{corollary}
\begin{proof}
	Note that by \eqref{eq:ComplexityParameter}, 
	\begin{equation*}
		\| \theta_{1} - \theta_0 \|_{\theta_0} = \left( \frac{\alpha \sqrt{\vartheta}}{\alpha \sqrt{\vartheta} - 1} - 1 \right) \| \theta_0 \|_{\theta_0} \leq \left( \frac{\alpha \sqrt{\vartheta}}{\alpha \sqrt{\vartheta} - 1} - 1 \right) \sqrt{\vartheta} = \frac{\sqrt{\vartheta}}{\alpha \sqrt{\vartheta}-1} < 1,
	\end{equation*}
	and similarly,
	\begin{equation*}
		\| \theta_{1} - \theta_0 \|_{\theta_1} = \frac{1}{\alpha\sqrt{\vartheta}} \| \theta_{1} \|_{\theta_{1}} \leq \frac{1}{\alpha \sqrt{\vartheta}} \sqrt{\vartheta} = \frac{1}{\alpha} \leq \frac{\sqrt{\vartheta}}{\alpha \sqrt{\vartheta}-1} < 1.
	\end{equation*}
	For these upper bounds, Theorem \ref{thm:MixingTimeSCCov} shows the result.
\end{proof}

In other words, the temperature scheme by Abernethy and Hazan allows mixing with a path length that has the same asymptotic complexity as the path length by Kalai and Vempala. This result shows that the $X_k$ and $Y_{jk}$ samples in Algorithm \ref{alg:KalaiVempala} approximately follow the correct distribution, if the path length $\ell$ is chosen as in \eqref{eq:WalkLengthSA}.


The main condition that needs to be satisfied before Corollary \ref{cor:MixingTimeSA} guarantees mixing is \eqref{eq:HitAndRunCovCondition}. As one would expect, approximating the covariance matrix of some distribution is possible if one can sample from this distribution.
However, with hit-and-run, we cannot sample exactly from the correct distribution, and the samples we do generate are not independent. Theorem \ref{ext1-thm:CovQuality} in \cite{badenbroek2018complexity} covers these pitfalls, but it requires a guarantee that the random walks have mixed well enough with the target distribution. Corollary \ref{cor:MixingTimeSA} provides such a guarantee. Letting $\lambda_{\min}(A)$ denote the smallest eigenvalue (with respect to the Euclidean inner product) of the matrix $A$, the following can thus be proven in the same manner as Theorem \ref{ext1-thm:CovQuality} in \cite{badenbroek2018complexity}.

\begin{theorem}
	\label{thm:CovQualitySA}
	Let $K \subseteq \mathbb{R}^n$ be a convex body, and let $\langle \cdot, \cdot \rangle$ be the Euclidean inner product. Suppose $K$ is contained in a Euclidean ball with radius $R > 0$. Let $p \in (0,1)$, $\varepsilon \in (0,1]$, and $\epsilon \geq 0$.
	Let $\theta_0, \theta_1 \in \mathbb{R}^n$ such that $\theta_1 = \theta_0/(1 - \frac{1}{\alpha\sqrt{\vartheta}})$, where $\vartheta$ is the complexity parameter of the entropic barrier over $K$, and $\alpha > 1+1/\sqrt{\vartheta}$. Suppose we have an invertible matrix $\widehat{\Sigma}(\theta_0)$ such that
	\begin{equation*}
		\frac{1}{1+\epsilon} v^\top \widehat{\Sigma}(\theta_0) v \leq v^\top \Sigma(\theta_0) v \leq (1+\epsilon) v^\top \widehat{\Sigma}(\theta_0) v \qquad \forall v \in \mathbb{R}^n.
	\end{equation*}
	Pick
	\begin{equation*}
		N \geq \frac{490 n^2}{p \varepsilon^2}, \qquad q \leq \frac{p \varepsilon^2}{49980 n^2 R^4} \lambda_{\min}(\Sigma(\theta_1))^2.
	\end{equation*}
	Let $X_0$ be a random starting point drawn from a Boltzmann distribution with parameter $\theta_0$. Let $Y^{(1)}, ..., Y^{(N)}$ be the end points of $N$ hit-and-run random walks applied to the Boltzmann distribution with parameter $\theta_1$ having starting point $X_0$, where the directions are drawn from a $\mathcal{N}(0,\widehat{\Sigma}(\theta_0))$-distribution, and each walk has length $\ell$ given by \eqref{eq:WalkLengthSA}. (Note that $\ell$ depends on $\epsilon$, $n$, and $q$.)
	Then, the empirical covariance matrix $\widehat{\Sigma}(\theta_1) \approx \Sigma(\theta_1)$ as defined in line  \ref{line:EmpiricalCovariance} of Algorithm \ref{alg:KalaiVempala} satisfies
	\begin{equation*}
		\mathbb{P} \left\{ \frac{1}{1+\varepsilon} v^\top \widehat{\Sigma}(\theta_1) v \leq v^\top \Sigma(\theta_1) v \leq (1+\varepsilon) v^\top \widehat{\Sigma}(\theta_1) v \qquad \forall v \in \mathbb{R}^n \right\} \geq 1-p.
	\end{equation*}
\end{theorem}
Note that a non-trivial lower bound on $\lambda_{\min}(\Sigma(\theta_1))$ was given in Theorem \ref{ext1-thm:LowerBoundSigma} in \cite{badenbroek2018complexity}. We see that a good approximation of $\widehat{\Sigma}(\theta_0)$ can be used to create a good approximation of $\widehat{\Sigma}(\theta_1)$. In other words, throughout the run of Algorithm \ref{alg:KalaiVempala}, the $X_k$ and $Y_{jk}$ samples are always from the desired distribution with high probability.

\section{Rate of convergence proof for Algorithm \ref{alg:KalaiVempala}}
\label{sec:Proof of Convergence}
We continue by proving that Algorithm \ref{alg:KalaiVempala} converges to the optimum in polynomial time.
The following result was established by Kalai and Vempala \cite{kalai2006simulated} for linear functions,
and extended from linear to convex functions $h: \mathbb{R}^n \to \mathbb{R}$ by De Klerk and Laurent \cite{deklerk2017comparison}.

\begin{theorem}[Corollary 1 in \cite{deklerk2017comparison}]
	\label{thm:KalaiVempalaExtended}
	Let $K \subseteq \mathbb{R}^n$ be a convex body.
	For any convex function $h: \mathbb{R}^n \to \mathbb{R}$, temperature $T > 0$, and $X \in K$ chosen according to a probability distribution on $K$ with density proportional to $x \mapsto e^{-h(x)/T}$, we have
	\begin{equation*}
		\mathbb{E}[h(X)] \leq nT + \min_{x \in K} h(x),
	\end{equation*}
	where
	\begin{equation*}
		\mathbb{E}[h(X)] := \frac{\int_K h(x) e^{-h(x)/T} \diff x}{\int_K e^{-h(x)/T} \diff x}.
	\end{equation*}
\end{theorem}

We are ready to prove that Algorithm \ref{alg:KalaiVempala} converges in polynomial time.
\begin{theorem}
	Let $K \subseteq \mathbb{R}^n$ be a convex body, and let $\langle \cdot, \cdot \rangle$ be the Euclidean inner product. Suppose $K$ is contained in a Euclidean ball with radius $R > 0$ and contains a Euclidean ball with radius $r > 0$. Denote the complexity parameter of the entropic barrier over $K$ by $\vartheta$, and fix some $\alpha > 1+1/\sqrt{\vartheta}$.
	Let $\bar{\epsilon} \in (0, 2R]$, $p \in (0,1)$, and $\epsilon = 1$. Assume $X_0$ is uniform on $K$ and we have an approximation $\widehat{\Sigma}(0)$ of $\Sigma(0)$ satisfying $\tfrac{1}{2} v^\top \widehat{\Sigma}(0) v \leq v^\top \Sigma(0) v \leq 2 v^\top \widehat{\Sigma}(0) v$ for all $v \in \mathbb{R}^n$.
	Consider Algorithm \ref{alg:KalaiVempala} with input
	\begin{align*}
		m &= \left\lceil (\alpha\sqrt{\vartheta}- \tfrac{1}{2}) \ln \left( \frac{2 n R}{p \bar{\epsilon}} \right) + 1 \right\rceil,\\
		N &= \left\lceil \frac{1000 n^2 m}{p} \right\rceil,\\
		q &= \left(\frac{p}{102000 m n^2 R^4} \right) \left(\frac{1}{4096} \right) \left(\frac{r}{n+1} \right)^4 \left( \frac{p \bar{\epsilon}}{8Rn} \right)^{8 \sqrt{\vartheta} + 4},
	\end{align*}
	and let $\ell$ be as in \eqref{eq:WalkLengthSA}.
	Then, with probability at least $1-p$, we have $\langle c, X_m \rangle - \min_{x \in K} \langle c,x \rangle \leq \bar{\epsilon}$. The number of membership oracle calls is $O^*(\vartheta^{3.5} n^5/p) = O^*(n^{8.5}/p)$.
\end{theorem}
\begin{proof}
	Throughout all iterations $k$ of Algorithm \ref{alg:KalaiVempala}, we want to maintain the conditions that $X_k$ is a sample from the Boltzmann distribution with parameter $\theta_k$, and
	\begin{equation}
		\tfrac{1}{2} v^\top \widehat{\Sigma}(\theta_k) v \leq v^\top \Sigma(\theta_k) v \leq 2 v^\top \widehat{\Sigma}(\theta_k) v \qquad \forall v \in \mathbb{R}^n,
		\label{eq:ProofComplexityInvariant}
	\end{equation}
	with high probability. We assume that these conditions hold in iteration $k-1$, and we will show that they then also hold for iteration $k$ with high probability. First, Corollary \ref{cor:MixingTimeSA} and Lemma \ref{lemma:DivineIntervention} guarantee that $X_k$ is sampled from the Boltzmann distribution with parameter $\theta_k$ with probability at least $1-q$. Moreover, noting that
	\begin{equation}
		m \geq (\alpha\sqrt{\vartheta}- \tfrac{1}{2}) \ln \left( \frac{2 n R}{p \bar{\epsilon}} \right) + 1
		\geq \frac{\ln \left( \frac{2 n R}{p \bar{\epsilon}} \right)}{\ln \left( \frac{\alpha\sqrt{\vartheta}}{\alpha \sqrt{\vartheta} - 1} \right)} + 1
		= \frac{\ln \left( \frac{p \bar{\epsilon}}{2 n R} \right)}{\ln \left( 1 - \frac{1}{\alpha \sqrt{\vartheta}} \right)} + 1,
		\label{eq:ProofComplexityM}
	\end{equation}
	we have
	\begin{equation*}
		\| \theta_k \| = \frac{1}{T_k} \leq \frac{1}{T_m} = \frac{1}{R} \left(1 - \frac{1}{\alpha\sqrt{\vartheta}}\right)^{1-m} \leq \frac{2n}{p\bar{\epsilon}},
	\end{equation*}
	and therefore Theorem \ref{ext1-thm:LowerBoundSigma} in \cite{badenbroek2018complexity} gives
	\begin{equation*}
		\lambda_{\min}(\Sigma(\theta_k))
		\geq \frac{1}{64} \left( \frac{1}{4 R \| \theta_k \|} \right)^{4 \sqrt{\vartheta} + 2} \left(\frac{r}{n+1} \right)^2
		\geq \frac{1}{64} \left( \frac{p \bar{\epsilon}}{8 R n} \right)^{4 \sqrt{\vartheta} + 2} \left(\frac{r}{n+1} \right)^2.
	\end{equation*}
	It then follows that
	\begin{equation*}
		q \leq \frac{p \frac{49}{100 m}}{49980 n R^4} \lambda_{\min}(\Sigma(\theta_k))^2 \quad \text{and} \quad N \geq \frac{490 n^2}{p \frac{49}{100 m}},
	\end{equation*}
	such that the conditions of Theorem \ref{thm:CovQualitySA} are satisfied. Hence, with probability $p \frac{49}{100 m}$, \eqref{eq:ProofComplexityInvariant} also holds for iteration $k$.
	
	By induction, $X_m$ is sampled from a Boltzmann distribution with parameter $\theta_m$ and \eqref{eq:WalkLengthSA} is satisfied for $k = m$, with probability at least
	\begin{equation*}
		1 - m\left(q + p \frac{49}{100 m}\right) \geq 1 - m\left(\frac{p}{100 m} + \frac{49 p}{100 m}\right) = 1-p/2,
	\end{equation*}
	by the union bound.
	It then follows from the Markov inequality, Theorem \ref{thm:KalaiVempalaExtended}, and \eqref{eq:ProofComplexityM}, that
	\begin{equation}
		\label{eq:MarkovInequality}
		\mathbb{P} \left\{ \langle c, X_m \rangle - \min_{x \in K} \langle c,x \rangle > \bar{\epsilon} \right\} \leq \frac{\mathbb{E}[\langle c, X_m \rangle - \min_{x \in K} \langle c,x \rangle]}{\bar{\epsilon}} \leq \frac{n T_m}{\bar{\epsilon}} \leq \frac{p}{2}.
	\end{equation}
	In conclusion, the total success probability of the algorithm is at least $1-p$.
	The number of oracle calls is
	\begin{equation*}
		O^*(m N \ell) = O^* \left( \sqrt{\vartheta} \frac{n^2\sqrt{\vartheta}}{p} n^3 \vartheta^{2.5} \right) = O^*(\vartheta^{3.5} n^5/p) = O^*(n^{8.5}/p),
	\end{equation*}
	where the final equality uses $\vartheta = n + o(n)$.
\end{proof}
We note that if a lower bound on $\lambda_{\min}(\Sigma(\theta))$ were found that was not exponential in $\vartheta$, then the algorithm complexity would improve by a factor $\vartheta^{2.5}$. Even then, we would still be a factor of $n^{1.5}$ away from the complexity $O^*(n^{4.5})$ by Kalai and Vempala. The reason for this gap is that they use the following corollary to a theorem by Rudelson \cite{rudelson1999random}.
\begin{theorem}[{\cite[Theorem A.1]{kalai2006simulated}}]
	\label{thm:KalaiVempalaA1}
	Let $\mu$ be an isotropic logconcave probability distribution over $\mathbb{R}^n$ with mean $0$, and let $\epsilon > 0, p > 1$. Then, there exists a number $N$ with
	\begin{equation}
		N = O \left( \frac{n p^2}{\epsilon^2} \log^2(n / \epsilon^2) \max\{p, \log n\} \right),
		\label{eq:KalaiVempalaN}
	\end{equation}
	such that for $N$ independent samples $X_1, ..., X_N$ from $\mu$ we have
	\begin{equation*}
		\mathbb{E}\left[ \Big\| \frac{1}{N} \sum_{i=1}^N X_i X_i^\top - I \Big\|^p \right] \leq \epsilon^p.
	\end{equation*}
\end{theorem}
However, the samples that are generated by hit-and-run sampling do not follow the target distribution $\mu$, and are not independent. As we have seen in e.g. Theorem \ref{thm:MixingTimeSCCov}, they are drawn from a distribution that has total variation distance $q > 0$ to $\mu$. Moreover, the samples were shown to be $6q$-independent by \cite[Lemma \ref{ext1-lemma:NearIndependenceHitAndRun}]{badenbroek2018complexity}, i.e. for any two hit-and-run samples $X$ and $Y$ and measurable $A,B \subseteq K$,
\begin{equation*}
	\left| \mathbb{P}\{ X \in A \wedge Y \in B \} - \mathbb{P}\{X \in A\} \mathbb{P}\{Y \in B\} \right| \leq 6q.
\end{equation*}
While Kalai and Vempala claim that Theorem \ref{thm:KalaiVempalaA1} can be extended to this setting without significantly changing \eqref{eq:KalaiVempalaN}, a formal proof is not given. In particular, it is not shown how the relaxation of the independence assumption can be aligned with Rudelson's proof.

There is another reason \eqref{eq:KalaiVempalaN} could change in the hit-and-run setting.
Theorem \ref{thm:KalaiVempalaA1} holds for isotropic probability distributions, so to apply it to a Boltzmann distribution with parameter $\theta \in \mathbb{R}^n$ over a convex body $K \subseteq \mathbb{R}^n$, we should first transform the $K$ by $\Sigma(\theta)^{-1/2}$. Consequently, the transformed body is not necessarily contained in a Euclidean ball with radius $R$ anymore; in the worst case, the radius of the new enclosing ball depends on the smallest eigenvalue of $\Sigma(\theta)$. Since we do not have a better bound on this eigenvalue than Theorem \ref{ext1-thm:LowerBoundSigma} in \cite{badenbroek2018complexity},
the diameter of $\Sigma(\theta)^{-1/2} K$ may be exponential in $n$.
The suggestions by Kalai and Vempala to prove a statement similar to Theorem \ref{thm:KalaiVempalaA1} in the hit-and-run setting require a bound on this diameter. If this bound grows exponentially in $n$, then $q$ should decrease exponentially in $n$ to compensate (see e.g. Lemma 2.7 and the proof of Theorem 5.9 from Kannan, L\'ovasz and Simonovits \cite{kannan1997random}). Such a small $q$ would contribute a polynomial factor of $n$ to the final algorithm complexity.

Assuming these issues can be overcome, Kalai and Vempala show the following.
\begin{theorem}[{\cite[Theorem 4.2]{kalai2006simulated}}]
	Let $t \geq 0$. With $N = O^*(t^3n)$ samples per phase in each of $m$ phases, every distribution encountered by the sampling algorithm is 160-isotropic with probability at least $1-m/2^t$.
\end{theorem}
Kalai and Vempala are then able to pick $t = O(\log(m/p))$, and can therefore claim $N = O^*(n)$.
We find in Theorem \ref{thm:CovQualitySA} that $N = O(n^2 / (p/m))$ samples are required, which is $O^*(n\sqrt{\vartheta}) = O^*(n^{1.5})$ worse than Kalai and Vempala's result.

One might still wonder how to generate a good estimate $\widehat{\Sigma}(0)$ of the uniform covariance matrix $\Sigma(0)$ to start Algorithm \ref{alg:KalaiVempala}.
For this purpose, the rounding the body procedure from Lov\'asz and Vempala \cite{lovasz2006simulated} is suitable. It is shown by Theorem 5.3 in \cite{lovasz2006simulated} that, with probability at least $1-1/n$, the uniform distribution over $\widehat{\Sigma}(0)^{-1/2} K$ is $2$-isotropic. As shown in the proof of Theorem \ref{thm:MixingTimeSCCov}, this is equivalent to
\begin{equation*}
	\tfrac{1}{2} v^\top \widehat{\Sigma}(0) v \leq v^\top \Sigma(0) v \leq 2 v^\top \widehat{\Sigma}(0) v \qquad \forall v \in \mathbb{R}^n,
\end{equation*}
which satisfies the starting condition for Algorithm \ref{alg:KalaiVempala}. The number of calls to the membership oracle for this procedure is $O^*(n^4)$.

%

\section{Numerical Examples on the Doubly Nonnegative Cone}
\label{sec:Numerical Examples on the Doubly Nonnegative Cone}

Let $\mathbb{S}^{m \times m}$ denote the space of real symmetric $m \times m$ matrices. We will test the method described above on the problem of determining whether some matrix $C \in \mathbb{S}^{m \times m}$ is copositive, i.e. $x^\top C x \geq 0$ for all $x \in \mathbb{R}^{m}_+$ (see Bomze \cite{bomze2012copositive} for a survey on copositive programming and its applications).
The standard SDP relaxation of this problem is the following:
\begin{equation}
	\inf \left\{  \langle C, X \rangle : \sum_{i=1}^{m} \sum_{j=1}^{m} X_{ij} \leq 1, X \geq 0, X \succeq 0 \right\},
	\label{eq:OriginalProblem}
\end{equation}
where $\langle \cdot, \cdot \rangle$ is the trace inner product. If the value of \eqref{eq:OriginalProblem} is nonnegative, the matrix $C$ is copositive. However, since we place a nonnegativity constraint on every element of the matrix $X$, the Newton system in every interior point iteration is of size $O(m^2 \times m^2)$, which quickly leads to impractical computation times (see e.g. Burer \cite{burer2010optimizing}).

Before we can apply Algorithm \ref{alg:KalaiVempala}, we need to translate \eqref{eq:OriginalProblem} to a problem over $\mathbb{R}^{m(m+1)/2}$. The approach is standard: for any $A = [A_{ij}]_{ij} \in \mathbb{S}^{m \times m}$, define
\begin{equation*}
	\svec(A) := (A_{11}, \sqrt{2} A_{12}, ..., \sqrt{2} A_{1m}, A_{22}, \sqrt{2} A_{23}, ..., \sqrt{2} A_{2m}, ..., A_{mm})^\top,
\end{equation*}
such that $\svec(A) \in \mathbb{R}^{m(m+1)/2}$. If $\mathbb{R}^{m(m+1)/2}$ is endowed with the Euclidean inner product, the adjoint of $\svec$ is defined for every $a \in \mathbb{R}^{m(m+1)/2}$ as
\begin{equation*}
	\smat(a) = \begin{bmatrix}
	a_1 & a_2 /\sqrt{2} & \dots & a_m / \sqrt{2}\\
	a_2 /\sqrt{2} & a_{m+1} & \dots & a_{2m-1} / \sqrt{2}\\
	\vdots & \vdots & \ddots & \vdots\\
	a_m / \sqrt{2} & a_{2m-1} / \sqrt{2} & \dots & a_{m(m+1)/2}
	\end{bmatrix},
\end{equation*}
such that $\smat(a) \in \mathbb{S}^{m \times m}$. Moreover, $\smat(\svec(A)) = A$ and $\svec(\smat(a)) = a$ for all $A$ and $a$. Now let $c = \svec(C)$. Problem \eqref{eq:OriginalProblem} is equivalent to the following problem over $\mathbb{R}^{m(m+1)/2}$:
\begin{equation}
	\inf \left\{ \langle c, x \rangle : \sum_{i=1}^{m} \sum_{j=1}^{m} (\smat(x))_{ij} \leq 1, x \geq 0, \smat(x) \succeq 0 \right\}.
	\label{eq:EntropicProblem}
\end{equation}
Note that membership of the feasible set of \eqref{eq:EntropicProblem} can be determined in $O(m^3)$ operations. Let $n = \frac{1}{2}m(m+1)$ be the number of variables in problem \eqref{eq:EntropicProblem}.

\subsection{Covariance Approximation}
\label{subsec:CovApproximation}
First, we investigate how the quality of the covariance approximation depends on the walk length for problem \eqref{eq:EntropicProblem}. We take $20,\! 000$ hit-and-run samples from the uniform distribution over the feasible set of \eqref{eq:EntropicProblem} with walk length $50,\! 000$ (directions are drawn from $\mathcal{N}(0,I)$ and the starting point is $\svec(mI+J) / (2e^\top \svec(mI+J))$, where $J$ is the all-ones matrix). These samples are used to create the estimate $\widehat{\Sigma}_0$.
Then, the experiment is repeated for walk lengths $\ell \leq 50,\! 000$ and sample sizes $N \leq 20,\! 000$. We refer to these new estimates as $\widehat{\Sigma}_{\ell, N} $. We would like
\begin{equation*}
	-\epsilon y^\top \widehat{\Sigma}_{\ell,N} y \leq y^\top (\widehat{\Sigma}_0 - \widehat{\Sigma}_{\ell,N}) y \leq \epsilon y^\top \widehat{\Sigma}_{\ell,N} y \qquad \forall y \in \mathbb{R}^n,
\end{equation*}
for some small $\epsilon > 0$. This is equivalent to
\begin{equation*}
	-\epsilon x^\top x \leq x^\top (\widehat{\Sigma}_{\ell,N}^{-1/2} \widehat{\Sigma}_0 \widehat{\Sigma}_{\ell,N}^{-1/2} - I) x \leq \epsilon x^\top x \qquad \forall x \in \mathbb{R}^n,
\end{equation*}
i.e. we would like the spectral radius of $\widehat{\Sigma}_{\ell,N}^{-1/2} \widehat{\Sigma}_0 \widehat{\Sigma}_{\ell,N}^{-1/2} - I$ to be at most $\epsilon$. Because the spectra of $\widehat{\Sigma}_{\ell,N}^{-1/2} \widehat{\Sigma}_0 \widehat{\Sigma}_{\ell,N}^{-1/2} - I$ and $\widehat{\Sigma}_{\ell,N}^{-1} \widehat{\Sigma}_0 - I$ are the same, we focus on the spectral radius $\rho(\widehat{\Sigma}_{\ell,N}^{-1} \widehat{\Sigma}_0 - I)$.

The result is shown in Figure \ref{fig:SampleSizeAndWalkLengthHessian}, where $m$ refers to the size of the matrices in \eqref{eq:OriginalProblem}. Hence, the covariance matrices in question are of size $\frac{1}{2}m(m+1) \times \frac{1}{2}m(m+1)$.

\begin{figure}[h]
	\centering
	
	\def\speca{10}
	\def\specb{20}
	
	\foreach \m in {5,10,15,20} {
		\begin{tikzpicture}
		\begin{loglogaxis}[ylabel={$\rho(\widehat{\Sigma}_{\ell,N}^{-1} \widehat{\Sigma}_0 - I)$}, xlabel={Sample size $N$}, small, title={$m = \m$}, minor y tick num = 4, legend pos = outer north east]
		\foreach \txt/\length in {100/100,200/200,500/500,1000/1000,2000/2000,5000/5000,{10,\! 000}/10000,{20,\! 000}/20000,{50,\! 000}/50000} {
			\addplot+[
			discard if not={ell}{\length},
			restrict expr to domain={\thisrow{n}}{\m:\m},
			] table[x=S,y=norm] {auto_res_cov.txt};
			
			\ifx\m\speca \addlegendentryexpanded{$\ell = \txt$} \else\fi
			\ifx\m\specb \addlegendentryexpanded{$\ell = \txt$} \else\fi
		}
		\end{loglogaxis}
		\end{tikzpicture}
	}
	\caption{Effect of sample size $N$ and walk length $\ell$ on quality of uniform covariance approximation}
	\label{fig:SampleSizeAndWalkLengthHessian}
\end{figure}
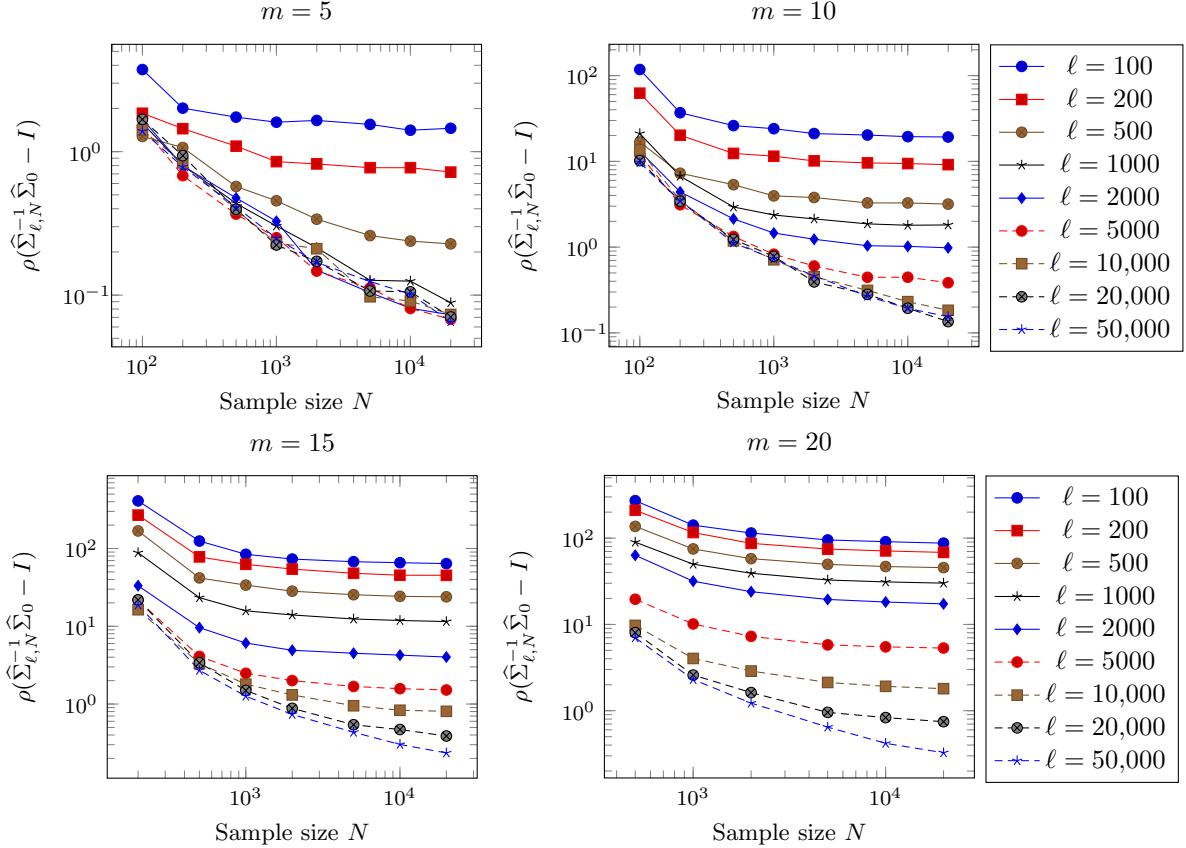

One major conclusion from Figure \ref{fig:SampleSizeAndWalkLengthHessian} is that the trajectory towards zero is relatively slow. To show that simply adding more samples with higher walk lengths will in practice not be feasible, we present the running times required to estimate a covariance matrix at the desired accuracy in Figure \ref{fig:RunTimeHessian}. Specifically, this figure shows the running times of the ``efficient'' combinations of $N$ and $\ell$: these are the combinations of $N$ and $\ell$ plotted in Figure \ref{fig:SampleSizeAndWalkLengthHessian} for which there are no $N'$ and $\ell'$ such that $N' \ell' \leq N \ell$ and $\rho(\widehat{\Sigma}_{\ell',N'}^{-1} \widehat{\Sigma}_0 - I) < \rho(\widehat{\Sigma}_{\ell,N}^{-1} \widehat{\Sigma}_0 - I)$. (The computer used has an Intel i7-6700 CPU with 16 GB RAM, and the code used six threads.) Figure \ref{fig:RunTimeHessian} shows that, even at low dimensions, approximating the covariance matrix to high accuracy will take an unpractical amount of time.

\begin{figure}
	\centering
	
	\begin{tikzpicture}
		\begin{loglogaxis}[xlabel={$\rho(\widehat{\Sigma}_{\ell,N}^{-1} \widehat{\Sigma}_0 - I)$}, ylabel={Running time (s)}, minor y tick num = 4, legend pos = outer north east]
		\foreach \m in {5,10,15,20} {
			\addplot+[
			discard if not={n}{\m},
			] table[x=rho,y=time] {auto_res_cov_time.txt};
			
			\addlegendentryexpanded{$m = \m$}
		}
		\end{loglogaxis}
	\end{tikzpicture}
	\caption{Running times required to find an approximation $\widehat{\Sigma}_{\ell,N}$ of the desired quality}
	\label{fig:RunTimeHessian}
\end{figure}
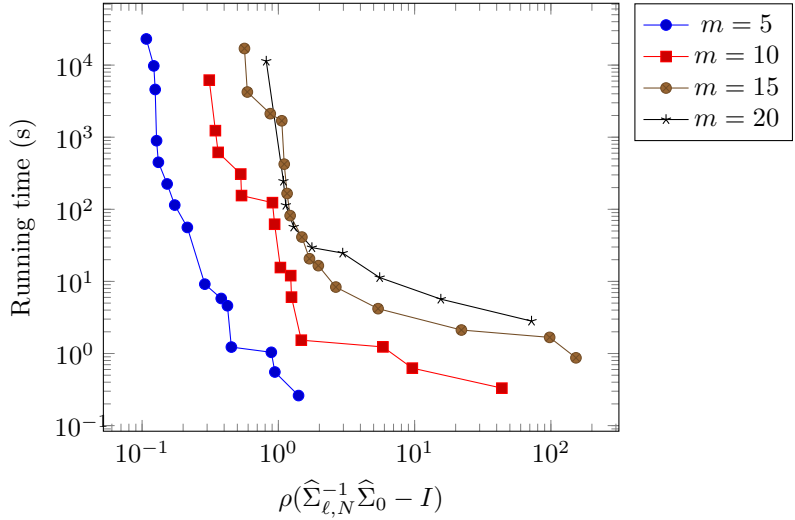


To show that the slow trajectory towards zero in Figure \ref{fig:SampleSizeAndWalkLengthHessian} is a result of convariance estimation's fundamental difficulty, we consider a simpler problem. We want to approximate the covariance matrix of the uniform distribution over the hypercube $[0,1]^n$ in $\mathbb{R}^n$. Note that the true covariance matrix of this distribution is known to be $\Sigma := \frac{1}{12} I$.

Again, we will use hit-and-run with varying walk lengths and sample sizes to generate samples from the uniform distribution over $[0,1]^n$, and compare the resulting covariance matrices $\widehat{\Sigma}_{\ell, N}$ with $\Sigma = \frac{1}{12} I$ (comparing against a covariance estimate based on hit-and-run samples as in Figure \ref{fig:SampleSizeAndWalkLengthHessian} yields roughly the same image). The result is shown in Figure \ref{fig:SampleSizeAndWalkLengthUni}.

\begin{figure}[h]
	\centering
	
	\foreach \n in {15,55,120,210} {
		\begin{tikzpicture}
		\begin{loglogaxis}[ylabel={$\rho(\widehat{\Sigma}_{\ell,N}^{-1} \Sigma - I)$}, xlabel={Sample size $N$}, small, title={$n = \n$}, minor y tick num = 4, legend pos=outer north east]
		\foreach \length in {0,100,200,500,1000} {
			\addplot+[
			discard if not={ell}{\length},
			restrict expr to domain={\thisrow{n}}{\n:\n},
			unbounded coords = discard,
			] table[x=S,y=norm]
			{auto_res_cube_cov.txt};
			
			\def\zero{0}
			\def\speca{55}
			\def\specb{210}
			
			\ifx\n\speca
				\ifx\length\zero
				\addlegendentryexpanded{Truly uniform}
				\else
				\addlegendentryexpanded{$\ell = \length$}
				\fi
			\else\fi
			\ifx\n\specb
				\ifx\length\zero
				\addlegendentryexpanded{Truly uniform}
				\else
				\addlegendentryexpanded{$\ell = \length$}
				\fi
			\else\fi
		}
		\end{loglogaxis}
		\end{tikzpicture}
	}
	\caption{Effect of sample size $N$ and walk length $\ell$ on quality of uniform covariance matrix approximation over the hypercube in $\mathbb{R}^n$}
	\label{fig:SampleSizeAndWalkLengthUni}
\end{figure}
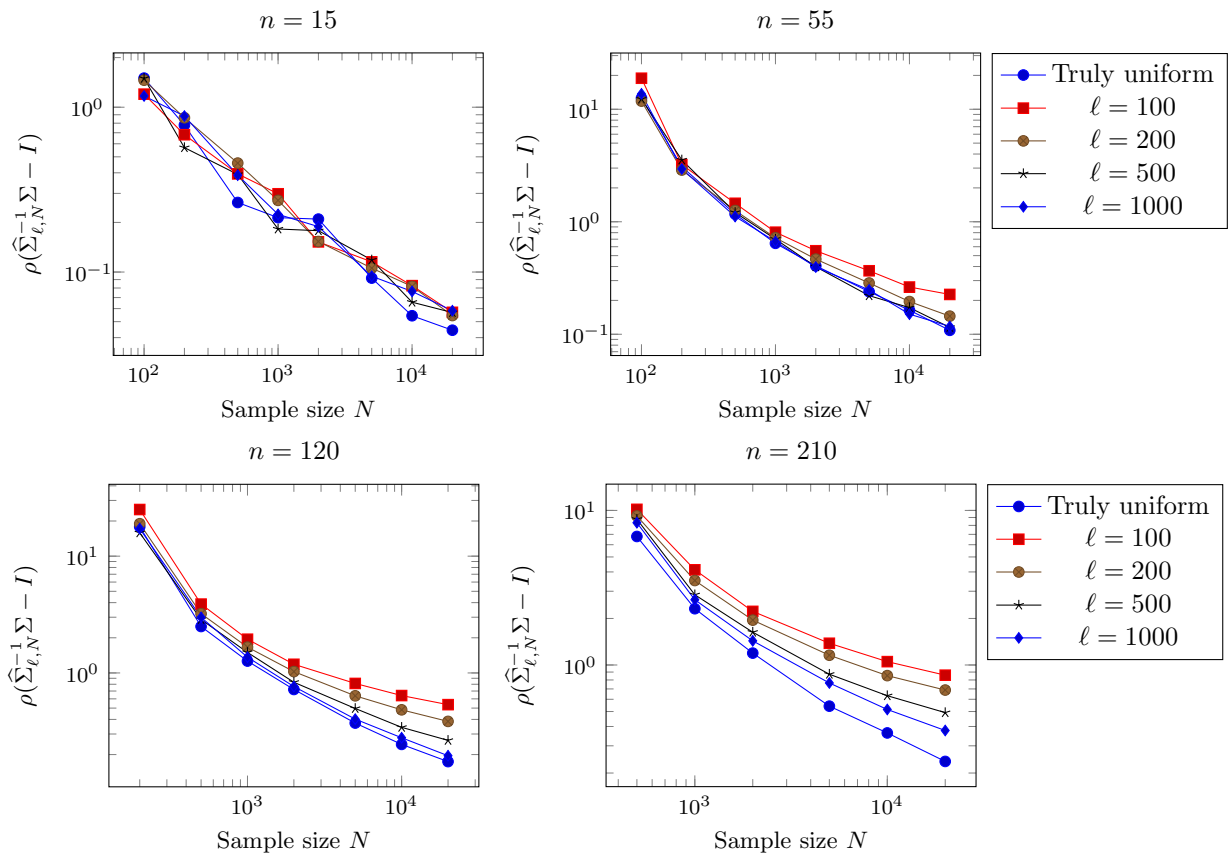

Figure \ref{fig:SampleSizeAndWalkLengthUni} shows a pattern similar to that of Figure \ref{fig:SampleSizeAndWalkLengthHessian}: as the problem size increases, the walk length should increase with the sample size to ensure the estimate is as good as the sample size can guarantee. While this progression towards zero may appear as slow, we do not have to know every entry of the covariance with high accuracy. Recall that we only use this covariance estimate in Algorithm \ref{alg:KalaiVempala} to generate hit-and-run directions. As such, it may suffice to have an estimate that roughly shows which directions are ``long'', and which ones are ``short''.

\subsection{Mean Approximation}
Next, we consider the problem of approximating the mean. Although it is not required for Algorithm \ref{alg:KalaiVempala} to approximate the mean of a Boltzmann distribution, such a mean does lie on the central path of the interior point method proposed by Abernethy and Hazan \cite{abernethy2015faster}.

We again take $20,\! 000$ hit-and-run samples from the uniform distribution over the feasible set of \eqref{eq:EntropicProblem} with walk length $50,\! 000$ (directions are drawn from $\mathcal{N}(0,I)$ and the starting point is $\svec(mI+J) / (2e^\top \svec(mI+J))$, where $J$ is the all-ones matrix). These samples are used to create the mean estimate $\widehat{x}_0$. Then, the experiment is repeated for walk lengths $\ell \leq 50,\! 000$ and sample sizes $N \leq 20,\! 000$. We refer to these new estimates as $\widehat{x}_{\ell, N}$. Using the approximation $\widehat{\Sigma}_0$ of the uniform covariance matrix from the previous section, we compute $\| \widehat{x}_0 - \widehat{x}_{\ell,N} \|_{\widehat{\Sigma}_0^{-1}}$ and plot the results in Figure \ref{fig:SampleSizeAndWalkLengthExpectation}.

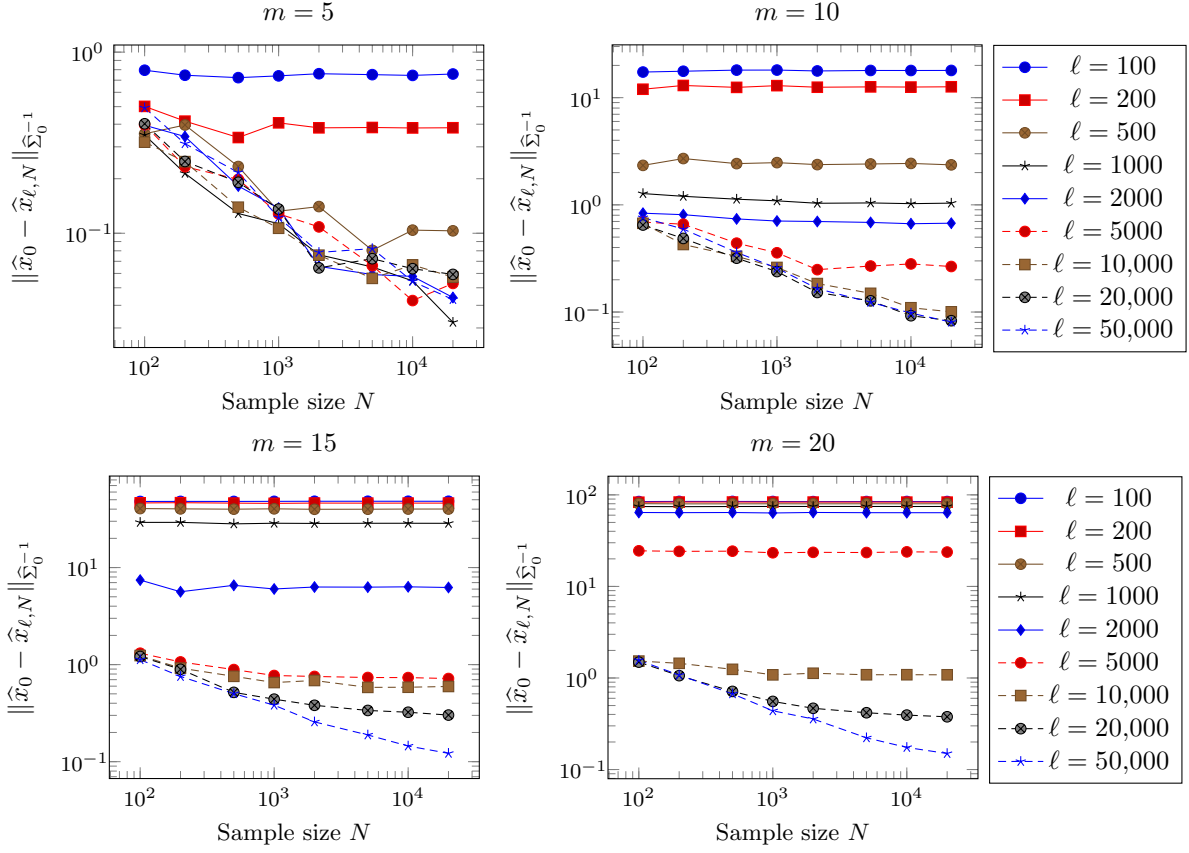
\begin{figure}[h]
	\centering
	
	\def\speca{10}
	\def\specb{20}
	
	\foreach \m in {5,10,15,20} {
		\begin{tikzpicture}
		\begin{loglogaxis}[ylabel={$\| \widehat{x}_0 - \widehat{x}_{\ell,N} \|_{\widehat{\Sigma}_0^{-1}}$}, xlabel={Sample size $N$}, small, title={$m = \m$}, minor y tick num = 4, ylabel near ticks,legend pos=outer north east]
		\foreach \txt/\length in {100/100,200/200,500/500,1000/1000,2000/2000,5000/5000,{10,\! 000}/10000,{20,\! 000}/20000,{50,\! 000}/50000} {
			\addplot+[
			discard if not={ell}{\length},
			restrict expr to domain={\thisrow{n}}{\m:\m},
			] table[x=S,y=norm] {auto_res_expec.txt};
			
			\ifx\m\speca \addlegendentryexpanded{$\ell = \txt$} \else\fi
			\ifx\m\specb \addlegendentryexpanded{$\ell = \txt$} \else\fi
		}
		\end{loglogaxis}
		\end{tikzpicture}
	}
	\caption{Effect of sample size $N$ and walk length $\ell$ on quality of uniform mean approximation}
	\label{fig:SampleSizeAndWalkLengthExpectation}
\end{figure}

The results are comparable to those in Figures \ref{fig:SampleSizeAndWalkLengthHessian} and \ref{fig:RunTimeHessian}. It will take an impractical amount of time before the mean estimate approximates the true mean well enough for practical purposes.

\subsection{Kalai-Vempala Algorithm}
\label{subsec:KalaiVempalaDNNCholesky}

The results from the previous two sections show that we should not hope to approximate the covariance matrix and sample mean with high accuracy in high dimensions. However, it is still insightful to verify if this is really required for Algorithm \ref{alg:KalaiVempala} to work in practice.

We therefore generated a random vector $c \in \mathbb{R}^{m(m+1)/2}$ as follows: if $C \in \mathbb{R}^{m \times m}$ is a matrix with all elements belonging to a standard normal distribution, then $C+C^\top +(\sqrt{2} - 2) \Diag (C)$ is a symmetric matrix whose elements all have variance 2. We then let
\begin{equation*}
	c = \frac{\svec(C+C^\top +(\sqrt{2} - 2) \Diag (C))}{\| \svec(C+C^\top +(\sqrt{2} - 2) \Diag (C))\|},
\end{equation*}
serve as the objective of our optimization problem \eqref{eq:EntropicProblem}. We can find the optimal solution $x_*$ with MOSEK 8.0 \cite{mosek8}, and then run Algorithm \ref{alg:KalaiVempala} with $\bar{\epsilon} = 10^{-3}$ and $p = 10^{-1}$. The final gap $\langle c, x_{\text{final}} - x_* \rangle$ is shown in Figure \ref{fig:DNNSAGapCholesky}. One can see that for practical sample sizes and walk lengths, the method does not converge to the optimal solution.
\begin{figure}[h]
	\centering
	
	\def\speca{10}
	\def\specb{20}
	
	\foreach \n in {5,10,15,20} {
		\begin{tikzpicture}
		\begin{loglogaxis}[ylabel={$\langle c, x_{\text{final}} - x_* \rangle$}, xlabel={Sample size $N$}, small, title={$n = \n$}, minor y tick num = 4, legend pos = outer north east, y tick label style={xshift=-0.6em}]
		\foreach \length in {5, 10, 20, 40, 80, 160, 320, 640, 1280} {
			\addplot+[
			discard if not={ell}{\length},
			restrict expr to domain={\thisrow{n}}{\n:\n},
			] table[x=S,y=gap] {auto_res_kalai_vempala_dnn_cholesky_gap.txt};
			
			\ifx\n\speca \addlegendentryexpanded{$\ell = \length$} \else\fi
			\ifx\n\specb \addlegendentryexpanded{$\ell = \length$} \else\fi
		}
		\end{loglogaxis}
		\end{tikzpicture}
	}
	\caption{Effect of sample size $N$ and walk length $\ell$ on the final gap of Algorithm \ref{alg:KalaiVempala}}
	\label{fig:DNNSAGapCholesky}
\end{figure}
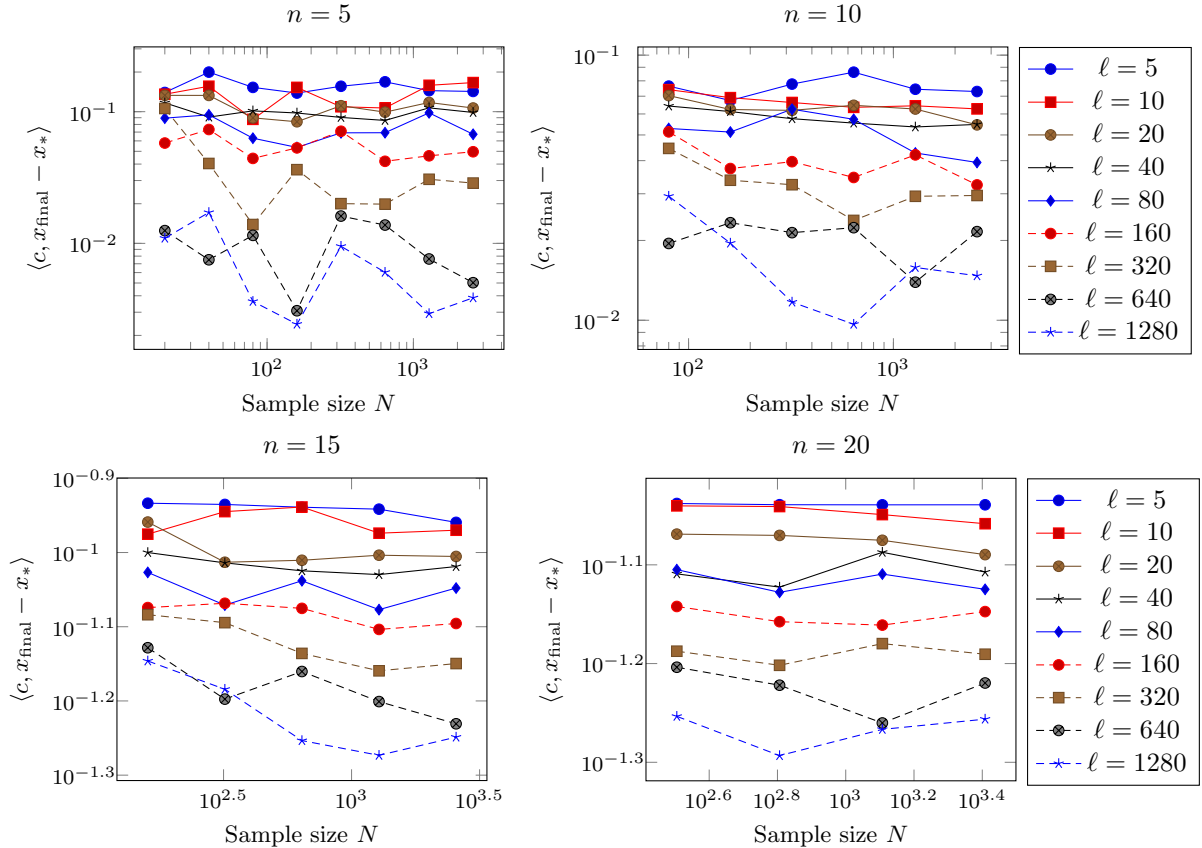

\subsection{Kalai-Vempala Algorithm with Acceleration Heuristic}
Keeping our findings above in mind, we propose 
the heuristic adaption of of Algorithm \ref{alg:KalaiVempala} presented in Algorithm \ref{alg:HeuristicKalaiVempala}.
The main modifications to Algorithm \ref{alg:KalaiVempala} are:
\begin{enumerate}
	\item Use the (centered) samples generated in the previous iteration as directions for hit-and-run in the current iteration. 
This would eliminate the need to estimate the covariance matrix of a distribution, only to then draw samples from that same distribution. 
Instead, we can also draw directions directly from the centered samples (cf.\ line 10 in Algorithm \ref{alg:HeuristicKalaiVempala}). Thus each sample is used to generate a hit-and-run direction with uniform probability.
	\item Start a random walk from the end point of the previous random walk, rather than from a common starting point. The idea here is that the random sample as a whole will exhibit less dependence, thus improving the approximation quality of the empirical distribution.
	\item As a starting point for the walk in some iteration $k$, use the sample mean from iteration $k-1$. While this does significantly change the distribution of the starting point, it concentrates more probability mass around the mean of the Boltzmann distribution with parameter $\theta_{k-1}$, such that the starting point of the random walk is likely already close to the mean of the Boltzmann distribution with parameter $\theta_k$. In a similar vein, we return the mean of the samples in the final iteration, not just one sample. This will not change the expected objective value of the final result, and will therefore also not affect the probabilistic guarantee that we derived in \eqref{eq:MarkovInequality} by Markov's inequality. However, using the mean does reduce the variance in the final solution.
\end{enumerate}
The modified algorithm is given in full in Algorithm \ref{alg:HeuristicKalaiVempala}.

\begin{algorithm}[ht!]
	\caption{Heuristic adaptation of Algorithm \ref{alg:KalaiVempala}}
	\label{alg:HeuristicKalaiVempala}
	\begin{algorithmic}[1]
		\Input
		unit vector $c \in \mathbb{R}^n$; membership oracle $\mathcal{O}_K: \mathbb{R}^n \to \{0,1\}$ of a convex body $K$; radius $R$ of Euclidean ball containing $K$;
		complexity parameter $\vartheta \leq n + o(n)$ of the entropic barrier over $K$; update parameter $\alpha > 1+ 1/\sqrt{\vartheta}$;
		error tolerance $\bar{\epsilon} > 0$; failure probability $p \in (0,1)$; number of hit-and-run steps $\ell \in \mathbb{N}$; number of samples $ N \in \mathbb{N}$;
		$y_1, ..., y_N \in K$ drawn randomly from the uniform distribution over $K$.
		\Statex
		
		\State $Y_{j0} \gets y_j$ for all $j \in \{1, ..., N\}$
		\State $X_0 \gets \frac{1}{N} \sum_{j=1}^N Y_{j0}$
		\State $\theta_0 \gets 0$
		\State $T_0 \gets \infty, T_1 \gets R$
		\State $k \gets 1$
		\While{$n T_{k-1} > \bar{\epsilon} p$} \label{line:TerminationCondition} 
		\State $\theta_k \gets -c / T_k$
		\State $Y_{0 k} \gets X_{k-1}$
		\For{$j \in \{1, ..., N\}$}
		\State \algmultilinetwo{Generate $Y_{jk}$ by applying hit-and-run sampling to the Boltzmann distribution with parameter $\theta_k$, starting the walk from $Y_{j-1,k}$, taking $\ell$ steps, drawing directions uniformly from 
		$\{ Y_{1, k-1} - X_{k-1}, ..., Y_{N, k-1} - X_{k-1} \}$} 
		\EndFor
		\State $X_k \gets \frac{1}{N} \sum_{j=1}^N Y_{jk}$
		\State $T_{k+1} \gets \min\{ R (1 - \frac{1}{\alpha \sqrt{\vartheta}})^{k}, R(1-\frac{1}{\sqrt{n}})^k\}$
		\State $k \gets k+1$
		\EndWhile
		\State \Return $X_{k-1}$
	\end{algorithmic}
\end{algorithm}

With these modifications implemented, we can no longer study the quality of the covariance matrix. Therefore, we will simply consider if the resulting optimization algorithm leads to a small error in the objective value. We solve the problem from Section \ref{subsec:KalaiVempalaDNNCholesky} with Algorithm \ref{alg:HeuristicKalaiVempala}. The results are given in Figure \ref{fig:DNNSAGapEmpirical}.

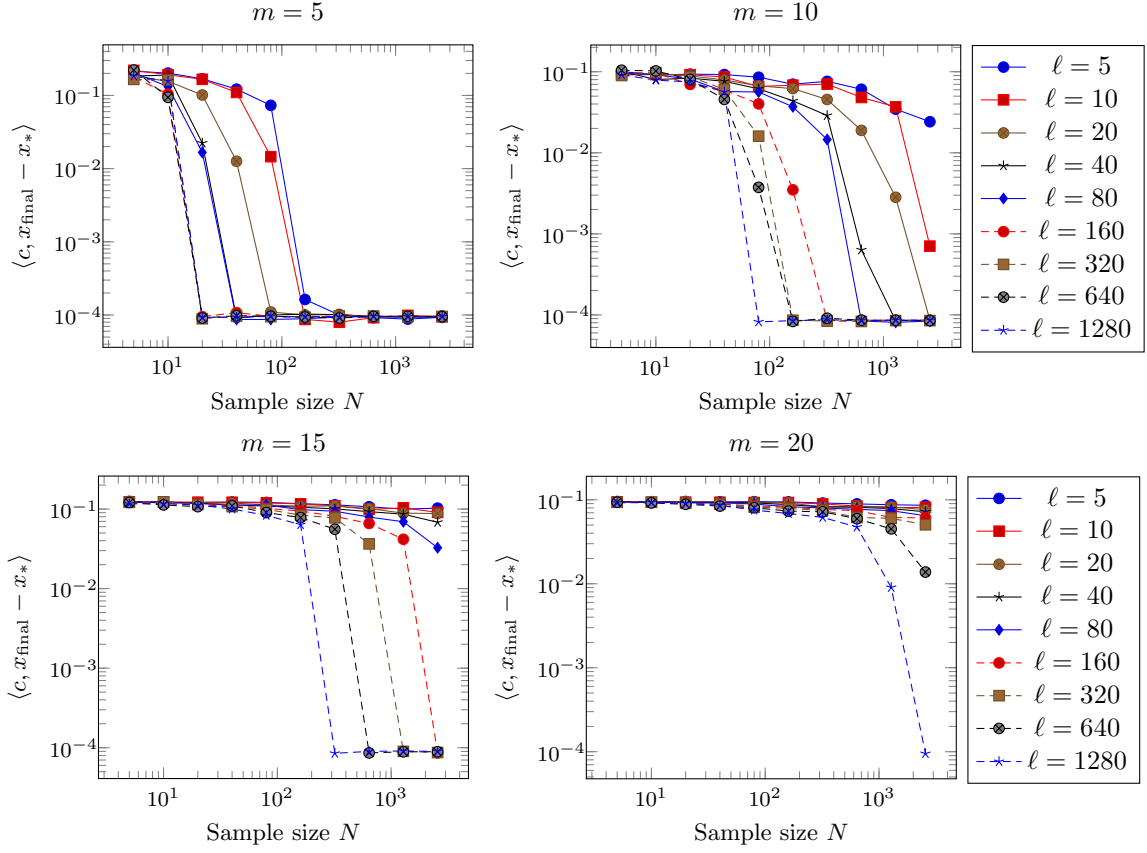
\begin{figure}[h]
	\centering
	
	\def\speca{10}
	\def\specb{20}
	
	\foreach \m in {5,10,15,20} {
		\begin{tikzpicture}
			\begin{loglogaxis}[ylabel={$\langle c, x_{\text{final}} - x_* \rangle$}, xlabel={Sample size $N$}, small, title={$m = \m$}, minor y tick num = 4, legend pos = outer north east]
			\foreach \length in {5, 10, 20, 40, 80, 160, 320, 640, 1280} {
				\addplot+[
				discard if not={ell}{\length},
				restrict expr to domain={\thisrow{n}}{\m:\m},
				] table[x=S,y=gap] {auto_res_kalai_vempala_dnn_empirical_gap.txt};
				
				\ifx\m\speca \addlegendentryexpanded{$\ell = \length$} \else\fi
				\ifx\m\specb \addlegendentryexpanded{$\ell = \length$} \else\fi
			}
			\end{loglogaxis}
%
		\end{tikzpicture}
	}
	\caption{Effect of sample size $N$ and walk length $\ell$ on the final gap of Algorithm \ref{alg:HeuristicKalaiVempala}}
	\label{fig:DNNSAGapEmpirical}
\end{figure}

For low dimensions in particular, the proposed changes seem to have a positive effect.
It can be seen from Figure \ref{fig:DNNSAGapEmpirical} that -- roughly speaking -- the final gap $\langle c, x_{\text{final}} - x_* \rangle$ takes values between two extremes. At one end, the method does not converge and the final gap is still of the order $10^{-1}$. At the other end, the method does converge to the optimum, such that the gap is of the order $10^{-4} = \bar{\epsilon} p$. Note that $\bar{\epsilon} p$ is exactly the size we would like the expected gap to have to guarantee that the gap is smaller than $\bar{\epsilon}$ with probability $1-p$ by Markov's inequality.
Whether we are at one end or the other depends on $N$ and $\ell$ being large enough compared to $m$. As a heuristic, we propose that
\begin{equation}
	N = \left\lceil n \sqrt{n} \right\rceil, \qquad \ell = \left\lceil n \sqrt{n} \right\rceil,
	\label{eq:HeuristicNEll}
\end{equation}
where $n = m(m+1)/2$ is the number of variables, are generally sufficient to ensure that the final gap is of the order $\bar{\epsilon} p$.

\section{Numerical Examples on the Copositive Cone}
\label{sec:Numerical Examples on the Copositive Cone}

We now turn our attention away from the doubly nonnegative cone, and towards the copositive cone.  Although deciding if a matrix is copositive is a co-NP-complete problem \cite{murty1987some}, there are a number of procedures to test for copositivity. Clearly, $A = [A_{ij}]_{ij} \in \mathbb{S}^{m \times m}$ is copositive if and only if
\begin{equation}
	\min \left\{ a^\top A a : e^\top a = 1, a \geq 0 \right\},
	\label{eq:CopositiveMemberQP}
\end{equation}
is nonnegative, where $e$ is the all-ones vector. Xia et al. \cite{xia2015globally} show that solving \eqref{eq:CopositiveMemberQP} is equivalent to solving
\begin{align}
	\begin{aligned}
		\min\, & - \nu\\
		\text{s.t.}\, & A a + \nu e - \eta = 0\\
		& e^\top a = 1\\
		& 0 \leq a \leq b\\
		& 0 \leq \eta \leq M(e - b)\\
		& b \in \{0,1\}^n,
	\end{aligned}
	\label{eq:CopositiveMemberMILP}
\end{align}
where $M = 2 m \max_{i,j \in \{1, ..., m\}} |A_{ij}|$. (To be precise, every optimal solution $(a, \nu, \eta)$ to \eqref{eq:CopositiveMemberMILP} gives an optimal solution $a$ to \eqref{eq:CopositiveMemberQP}, and these two problems have the same optimal values.)
Note that we are generally not interested in solving \eqref{eq:CopositiveMemberMILP} to optimality: it suffices to find a feasible solution of \eqref{eq:CopositiveMemberMILP} with a negative objective value, or confirm that no such solution exists. For the majority of the matrices encountered by Algorithm \ref{alg:HeuristicKalaiVempala} applied to our test sets described below, this could be checked quickly.

\subsection{Separating from the Completely Positive Cone}
\label{subsec:CompletelyPositive}
Recall that a matrix $A \in \mathbb{S}^{m \times m}$ is completely positive if $A = B B^\top$ for
some $B \geq 0$. It is easily seen that optimization problems over the completely positive cone can
be relaxed as optimization problems over the doubly nonnegative cone.
 To strengthen this relaxation, one could add a cutting plane separating the optimal solution $Y$ of the doubly nonnegative
  relaxation from the completely positive cone. This is listed as an open (computational) problem by Berman et al. \cite[Section 5]{berman2015open},
 who note that the problem of generating such a cut has only been answered for specific structures of $Y$,
   including $5 \times 5$ matrices \cite{burer2013separation}. In general,
    such a cut could be generated for a doubly nonnegative matrix $Y$ by the copositive program
\begin{equation}
	\inf \left\{ \langle Y, X \rangle : \langle X, X \rangle \leq 1, X \text{ copositive} \right\}.
	\label{eq:CompletelyPositiveSeparation}
\end{equation}
Below, we solve this problem for $6 \times 6$ matrices, by way of example.

To generate test instances, we are interested in matrices on the boundary of the $6 \times 6$ doubly nonnegative cone.
The extreme rays
 of this cone are described by Ycart \cite[Proposition 6.1]{ycart1982extremales}. We generate random instances from the class of matrices described under case 3, graph 4 in Proposition 6.1 in \cite{ycart1982extremales}. These matrices are (up to permutation of the indices) doubly nonnegative matrices $Y = [Y_{ij}]_{ij}$ with rank 3 satisfying $Y_{i,i+1} = 0$ for $i = 1, ..., 5$. To generate such a matrix, we draw the elements of two vectors $v_1, v_2 \in \mathbb{R}^6$ and the first element $(v_3)_1 \in \mathbb{R}$ of a vector $v_3 \in \mathbb{R}^6$ from a Poisson distribution with rate $1$, and multiply each of these elements with $-1$ with probability $\frac{1}{2}$.
The remaining elements of $v_3$ are then chosen such that $Y = \sum_{k=1}^3 v_k v_k^\top$ satisfies $Y_{i,i+1} = 0$ for $i = 1, ..., 5$. This procedure is repeated if the matrix $Y$ is not doubly nonnegative, or if BARON 15 \cite{tawarmalani2005polyhedral} could find a nonnegative matrix $B \in \mathbb{R}^{6\times 9}$ such that $Y = BB^\top$ in less than 30 seconds (for the cases where such a decomposition could be found, BARON terminated in less than a second). Thus, we are left with doubly nonnegative matrices for which it cannot quickly be shown that they are completely positive.

For ten of such randomly generated matrices
(see Appendix \ref{app:ExtremalDNNMatrices}), the optimal value of Algorithm \ref{alg:HeuristicKalaiVempala} applied to
 \eqref{eq:CompletelyPositiveSeparation} is given in Table \ref{tab:Separation}. This table shows the normalized objective value $\langle Y / \| Y\|, X^* \rangle$, where $Y$ is a doubly nonnegative matrix as described above,
  and $X^*$ is the final solution returned by Algorithm \ref{alg:HeuristicKalaiVempala}.
\begin{table}
	\centering
	\begin{tabular}{lllll}
		\toprule
		& \multicolumn{2}{c}{Final objective value (normalized)} & \multicolumn{2}{c}{Oracle calls}\\
		\cmidrule(lr){2-3} \cmidrule(lr){4-5}
		Name & Algorithm \ref{alg:HeuristicKalaiVempala} & Ellipsoid method & Algorithm \ref{alg:HeuristicKalaiVempala} & Ellipsoid method\\
  		\midrule
\texttt{extremal\_rand\_1}  & -7.626893e-03 &-7.667645e-03& 8.766473e+06 & 3.152000e+03\\
\texttt{extremal\_rand\_2}  & -1.983630e-02 &-1.987634e-02& 9.073317e+06 & 3.412000e+03\\
\texttt{extremal\_rand\_3}  & -3.591875e-02 &-3.596345e-02& 9.334264e+06 & 3.835000e+03\\
\texttt{extremal\_rand\_4}  & -9.937402e-03 &-9.980087e-03& 8.830209e+06 & 3.147000e+03\\
\texttt{extremal\_rand\_5}  & -5.897273e-03 &-5.940056e-03& 8.628287e+06 & 2.957000e+03\\
\texttt{extremal\_rand\_6}  & -4.303956e-02 &-4.307761e-02& 9.438518e+06 & 4.024000e+03\\
\texttt{extremal\_rand\_7}  & -2.411010e-02 &-2.415651e-02& 9.179767e+06 & 3.708000e+03\\
\texttt{extremal\_rand\_8}  & -6.822593e-02 &-6.826558e-02& 9.641288e+06 & 4.277000e+03\\
\texttt{extremal\_rand\_9}  & -4.232229e-02 &-4.236829e-02& 9.416909e+06 & 3.981000e+03\\
\texttt{extremal\_rand\_10} & -2.962993e-02 &-2.967333e-02& 9.236507e+06 & 3.743000e+03\\
		\bottomrule
	\end{tabular}
	\caption{Objective values returned by Algorithm \ref{alg:HeuristicKalaiVempala} and by the Ellipsoid method, applied to \eqref{eq:CompletelyPositiveSeparation}. Algorithm  \ref{alg:HeuristicKalaiVempala} was run with
 $\bar{\epsilon} = 10^{-3}$ and $p = 0.1$, and $N$ and $\ell$ as in \eqref{eq:HeuristicNEll}. The Ellipsoid method was run  with error tolerance $10^{-4}$.}
	\label{tab:Separation}
\end{table}
Note that in all cases, Algorithm \ref{alg:HeuristicKalaiVempala} succeeded in finding a copositive matrix $X^*$ such that $\langle Y, X^* \rangle < 0$, which means a cut separating $Y$ from the completely positive matrices was found.

Note that solving the MILP \eqref{eq:CopositiveMemberMILP} for a matrix $A$ that is not copositive
yields a hyperplane separating $A$ from the copositive cone. Thus, we can also solve problem \eqref{eq:CompletelyPositiveSeparation} with the ellipsoid method of Yudin and Nemirovski \cite{Yudin-Nemirovski 1976}, for example.
 For the sake of comparison, the results of the Ellipsoid method are also included in Table \ref{tab:Separation}. Note, in particular, that the number of oracle calls in Table \ref{tab:Separation} is
 several orders of magnitude smaller for the Ellipsoid method.
%

\section{Conclusion}
We have shown that Kalai and Vempala's algorithm \cite{kalai2006simulated} returns a solution which is near-optimal for \eqref{eq:MainProblem} with high
probability in polynomial time, when the temperature update \eqref{AH type temp schedule} is used.
The main drawback to using the algorithm in practice, is the large number of samples (i.e.\ membership oracle calls) required.
As a result, in our tests the Ellipsoid method outperformed Algorithm \ref{alg:HeuristicKalaiVempala} by a large margin.
Thus, based on our experiments, one would favor polynomial-time cutting plane methods like the Ellipsoid method, or more sophisticated alternatives
as described e.g.\ in \cite{faster cutting plane}.
In order to obtain a practically viable variant of the  Kalai-Vempala algorithm, one would have to improve the sampling process greatly, or utilize massive parallelism to speed up the hit-and-run sampling.

\subsubsection*{Acknowledgements}
The authors would like to thank S\'ebastien Bubeck, Osman G\"uler, and Levent Tun\c{c}el for valuable discussions about the complexity parameter of the entropic barrier.

\appendix

\section{The complexity parameter of the entropic barrier for the Euclidean ball}
\label{appendix:complexity parameter ball}

Let $B_n := \{ x \in \mathbb{R}^n: \| x \|^2 \leq 1 \}$ be the unit ball in $\mathbb{R}^n$ with respect to the inner product $\langle \cdot, \cdot \rangle$. We are interested in the complexity parameter $\vartheta$ of the entropic barrier on $B_n$. We will follow the notation from \cite{badenbroek2018complexity}.

Bubeck and Eldan \cite[Lemma 1(iii)]{bubeck2018entropic} show that the gradient $g^*$ of the entropic barrier is a bijection of $B_n$ to $\mathbb{R}^n$. Therefore,
\begin{equation*}
	\vartheta = \sup_{x \in B_n} (\| g_x^*(x) \|_x^*)^2 = \sup_{x \in B_n} \langle g^*(x), H^*(x)^{-1} g^*(x) \rangle = \sup_{\theta \in \mathbb{R}^n} \langle \theta, H(\theta) \theta \rangle,
\end{equation*}
where the final equality uses $H^*(x)^{-1} = H(\theta(x))$ (see \cite[Lemma 1(iv)]{bubeck2018entropic}).
Recall that $H(\theta)$ is the covariance operator of the Boltzmann distribution with parameter $\theta$. Then,
\begin{equation*}
	\langle \theta, H(\theta) \theta \rangle = \mathbb{E}_\theta[ \langle X - \mathbb{E}_\theta[X], \theta \rangle^2 ] = \mathbb{E}_\theta[ \langle \theta, X \rangle^2 ] - \langle \theta, \mathbb{E}_\theta[X] \rangle^2.
\end{equation*}

From now on, we will let $\langle \cdot, \cdot \rangle$ be the Euclidean inner product.
For every $\theta \in \mathbb{R}^n$, there exists a rotation matrix $Q$ with $| \det Q |= 1$ such that $\langle \theta, Q y \rangle = \| \theta \| y_1$ for all $y \in \mathbb{R}^n$. Using the fact that the volume of an $(n-1)$-dimensional ball with radius $r$ is $r^{n-1}$ times some factor depending only on $n$, we see that
\begin{equation*}
	\langle \theta, \mathbb{E}_\theta[X]\rangle = \frac{\int_{B_n} \langle \theta, x \rangle e^{\langle \theta, x \rangle} \diff x}{\int_{B_n} e^{\langle \theta, x \rangle} \diff x} = \frac{\int_{B_n} \| \theta \| y_1 e^{\| \theta \| y_1} \diff y}{\int_{B_n} e^{\| \theta \| y_1} \diff y}
	= \frac{\int_{-1}^1 \| \theta \| y_1 (\sqrt{1 - y_1^2})^{n-1} e^{\| \theta \| y_1} \diff y_1}{\int_{-1}^1 (\sqrt{1 - y_1^2})^{n-1} e^{\| \theta \| y_1} \diff y_1}.
\end{equation*}
The final expression cannot be computed in closed form, but it can be approximated numerically for fixed $\| \theta \|$. Similarly,
\begin{equation*}
	\mathbb{E}_\theta[ \langle \theta, X \rangle^2 ] = \frac{\int_{B_n} \langle \theta, x \rangle^2 e^{\langle \theta, x \rangle} \diff x}{\int_{B_n} e^{\langle \theta, x \rangle} \diff x} = \frac{\int_{B_n} \| \theta \|^2 y_1^2 e^{\| \theta \| y_1} \diff y}{\int_{B_n} e^{\| \theta \| y_1} \diff y}
	= \frac{\int_{-1}^1 \| \theta \|^2 y_1^2 (\sqrt{1 - y_1^2})^{n-1} e^{\| \theta \| y_1} \diff y_1}{\int_{-1}^1 (\sqrt{1 - y_1^2})^{n-1} e^{\| \theta \| y_1} \diff y_1}.
\end{equation*}

Numerical approximation of $\mathbb{E}_\theta[ \langle \theta, X \rangle^2 ] - \langle \theta, \mathbb{E}_\theta[X] \rangle^2$ for different values of $n$ and $\| \theta \|$ yields Figure \ref{fig:ApproxVarthetaBall}.
 This figure suggests that $\vartheta = \frac{1}{2} (n+1)$.

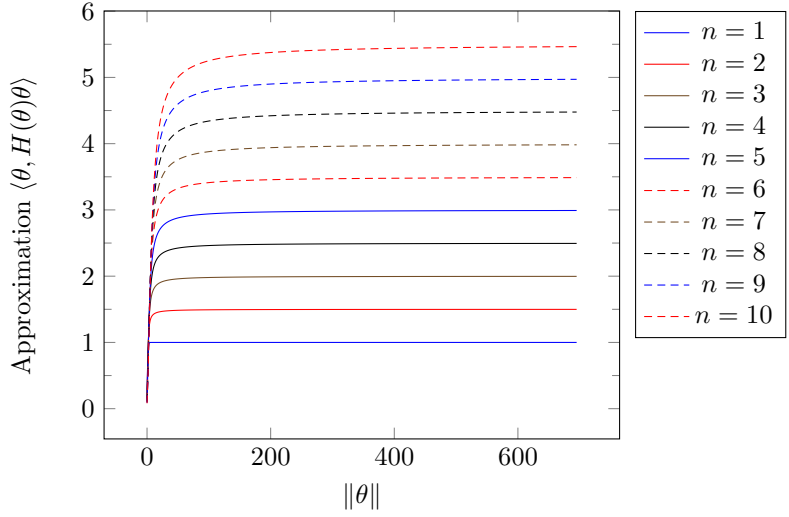
\begin{figure}
	\centering
	\begin{tikzpicture}
		\begin{axis}[legend pos = outer north east, minor y tick num=1, ytick distance = 1, ylabel = {Approximation $\langle \theta, H(\theta) \theta \rangle$}, xlabel = {$\|\theta \|$}]
			\foreach \n in {1,...,10} {
				\addplot+[no marks] table[x expr=\coordindex, y index=\n] {entropic_barrier_unit_ball_data.txt};
				\addlegendentryexpanded{$n = \n$};
			}
		\end{axis}
	\end{tikzpicture}
	\caption{Numerical approximation of $\langle \theta, H(\theta) \theta \rangle$}
	\label{fig:ApproxVarthetaBall}
\end{figure}

\section{Extremal Doubly Nonnegative Matrix Examples}
\label{app:ExtremalDNNMatrices}
Below are the ten randomly generated extreme points of the $6 \times 6$ doubly nonnegative cone that are used in Section \ref{subsec:CompletelyPositive}. These matrices can be strictly separated from the completely positive cone.

\begingroup
\allowdisplaybreaks
\begin{align*}
	\texttt{extremal\_rand\_1} &=
	\begin{bmatrix}
		2 & 0 & 6 & 0 & 1 & 2\\
0 & 6 & 0 & 8 & 1 & 2\\
6 & 0 & 18 & 0 & 3 & 6\\
0 & 8 & 0 & 11 & 0 & 3\\
1 & 1 & 3 & 0 & 6 & 0\\
2 & 2 & 6 & 3 & 0 & 3\\

	\end{bmatrix} &
	\texttt{extremal\_rand\_2} &=
	\begin{bmatrix}
		2 & 0 & 2 & 3 & 1 & 3\\
0 & 2 & 0 & 3 & 1 & 1\\
2 & 0 & 3 & 0 & 2 & 1\\
3 & 3 & 0 & 18 & 0 & 12\\
1 & 1 & 2 & 0 & 2 & 0\\
3 & 1 & 1 & 12 & 0 & 9\\

	\end{bmatrix} \\
	\texttt{extremal\_rand\_3} &=
	\begin{bmatrix}
		12 & 0 & 4 & 2 & 0 & 2\\
0 & 2 & 0 & 2 & 1 & 2\\
4 & 0 & 2 & 0 & 1 & 0\\
2 & 2 & 0 & 3 & 0 & 3\\
0 & 1 & 1 & 0 & 2 & 0\\
2 & 2 & 0 & 3 & 0 & 3\\

	\end{bmatrix} &
	\texttt{extremal\_rand\_4} &=
	\begin{bmatrix}
		2 & 0 & 2 & 2 & 2 & 4\\
0 & 8 & 0 & 4 & 4 & 8\\
2 & 0 & 3 & 0 & 4 & 0\\
2 & 4 & 0 & 8 & 0 & 16\\
2 & 4 & 4 & 0 & 8 & 0\\
4 & 8 & 0 & 16 & 0 & 32\\

	\end{bmatrix} \\
	\texttt{extremal\_rand\_5} &=
	\begin{bmatrix}
		5 & 0 & 5 & 0 & 5 & 3\\
0 & 6 & 0 & 10 & 1 & 18\\
5 & 0 & 5 & 0 & 5 & 3\\
0 & 10 & 0 & 20 & 0 & 42\\
5 & 1 & 5 & 0 & 6 & 0\\
3 & 18 & 3 & 42 & 0 & 99\\

	\end{bmatrix} &
	\texttt{extremal\_rand\_6} &=
	\begin{bmatrix}
		3 & 0 & 3 & 4 & 0 & 4\\
0 & 6 & 0 & 2 & 6 & 2\\
3 & 0 & 11 & 0 & 4 & 0\\
4 & 2 & 0 & 8 & 0 & 8\\
0 & 6 & 4 & 0 & 8 & 0\\
4 & 2 & 0 & 8 & 0 & 8\\

	\end{bmatrix} \\
	\texttt{extremal\_rand\_7} &=
	\begin{bmatrix}
		14 & 0 & 4 & 8 & 2 & 16\\
0 & 6 & 0 & 4 & 2 & 8\\
4 & 0 & 8 & 0 & 8 & 0\\
8 & 4 & 0 & 8 & 0 & 16\\
2 & 2 & 8 & 0 & 9 & 0\\
16 & 8 & 0 & 16 & 0 & 32\\

	\end{bmatrix} &
	\texttt{extremal\_rand\_8} &=
	\begin{bmatrix}
		6 & 0 & 4 & 2 & 0 & 2\\
0 & 5 & 0 & 2 & 2 & 2\\
4 & 0 & 6 & 0 & 2 & 0\\
2 & 2 & 0 & 2 & 0 & 2\\
0 & 2 & 2 & 0 & 2 & 0\\
2 & 2 & 0 & 2 & 0 & 2\\

	\end{bmatrix} \\
	\texttt{extremal\_rand\_9} &=
	\begin{bmatrix}
		2 & 0 & 2 & 0 & 4 & 2\\
0 & 2 & 0 & 2 & 2 & 0\\
2 & 0 & 2 & 0 & 4 & 2\\
0 & 2 & 0 & 3 & 0 & 2\\
4 & 2 & 4 & 0 & 14 & 0\\
2 & 0 & 2 & 2 & 0 & 6\\

	\end{bmatrix} &
	\texttt{extremal\_rand\_10} &=
	\begin{bmatrix}
		2 & 0 & 2 & 2 & 0 & 2\\
0 & 2 & 0 & 2 & 2 & 2\\
2 & 0 & 3 & 0 & 1 & 0\\
2 & 2 & 0 & 8 & 0 & 8\\
0 & 2 & 1 & 0 & 3 & 0\\
2 & 2 & 0 & 8 & 0 & 8\\

	\end{bmatrix} \\
\end{align*}
\endgroup

\bibliographystyle{abbrv}

\end{document}